\documentclass{article}

\usepackage{amsmath, amsfonts, amsthm, amssymb,cite,color,listings,graphicx,subcaption}
\usepackage[margin = 1in]{geometry}
\usepackage[capitalise]{cleveref}

\newtheorem{thm}{Theorem}

\newtheorem{ex}[thm]{Example}
\newtheorem{lemma}[thm]{Lemma}
\newtheorem{prop}[thm]{Proposition}
\newtheorem{cor}[thm]{Corollary}

\newtheorem{defn}[thm]{Definition}

\newcommand{\ignore}[1]{}

\newcommand{\N}{\mathbb{N}}
\newcommand{\R}{\mathbb{R}}

\newcommand{\Z}{\mathbb{Z}}
\newcommand{\C}{\mathbb{C}}
\newcommand{\B}{\mathbb{B}}

\newcommand{\Int}{\text{Int }}

\newcommand{\graph}{\text{graph}}

\renewcommand{\Re}{\text{Re}}
\renewcommand{\Im}{\text{Im}}

\DeclareMathOperator{\sech}{sech}

\title{Rate-induced Tipping in Discrete-time Dynamical Systems}

\author{Claire Kiers}

\begin{document}

\maketitle

\begin{abstract}
We develop a definition of rate-induced tipping (R-tipping) in discrete-time dynamical systems (maps) and prove results giving conditions under which R-tipping will or will not happen. Specifically, we study (possibly non-invertible) maps with a time-varying parameter subject to a parameter shift. We show that each stable path has a unique associated solution (a local pullback attractor) which stays near the path for all negative time. When the parameter changes slowly, this local pullback attractor stays near the path for all time, but if the parameter changes quickly, the local pullback attractor may move away from the path in positive time; this is the phenomenon of R-tipping. We demonstrate that forward basin stability is an insufficient condition to prevent R-tipping in maps of any dimension but that forward inflowing stability is sufficient. Furthermore, we show that R-tipping will happen when there is a certain kind of forward basin instability, and we prove precisely what happens to the local pullback attractor as the rate of the parameter change approaches infinity. We then highlight the differences between discrete- and continuous-time systems by showing that when a map is obtained by discretizing a flow, the pullback attractors for the map and flow can have dramatically different behavior; there may be R-tipping in one system but not in the other. We finish by applying our results to demonstrate R-tipping in the 2-dimensional Ikeda map.
\end{abstract}

\section{Introduction}

In this paper, we consider what it means for there to be rate-induced tipping (R-tipping) in a discrete-time dynamical system (map or difference equation) with time-varying parameters. Rate-induced tipping has already been studied extensively in continuous-time dynamical systems, or flows (\cite{sebas}, \cite{pete}, \cite{kiers}), but one cannot simply deduce results about R-tipping for maps from flows. To begin with, not all maps come from flows; flows are invertible while many maps are non-invertible. Furthermore, even if a map is obtained from a flow (say, by a taking a Poincar\'e section or by evaluating the flow at evenly-spaced discrete time steps), the parameter change affects the flow and the map differently. In the flow the parameter changes constantly, while in the map the parameter is fixed for each evaluation of the map. As a result, solutions to the corresponding map and flow can look very different. Since some physical processes are better modelled with maps than with flows (such as the Ikeda map of \cref{section:ikeda}), we endeavor here to establish the basic theory of R-tipping in maps.

Rate-induced tipping is, roughly, a drastic change in the behavior of a system due to quickly-changing parameters. Imagine trying to pull a tablecloth out from under a set of dishes on a table. If the tablecloth is pulled slowly, it will carry all the dishes with it, but (theoretically) if the tablecloth is pulled quickly enough the dishes will be left behind on the table. Here we get two distinct outcomes: dishes come with the tablecloth vs. dishes get left behind. The deciding factor between these outcomes is not how {\em far} the tablecloth moves, but how {\em fast}. Likewise, with R-tipping, the end behavior of solutions is determined not by how {\em much} the parameters change, but how {\em quickly}.

R-tipping was introduced in \cite{compost} and compared against other kinds of tipping (bifurcation and noise) in \cite{pete}. Bifurcation-induced tipping happens due to a bifurcation in the system (say, the annihilation of a stable fixed point) and is the result of parameters changing too {\em much}. With rate-induced tipping, there is no such bifurcation in the system to explain the sudden change in behavior. Noise-induced tipping happens as a result of noise in the system, which rate-induced tipping does not require, although there can be interplay between the two phenomena, as studied in \cite{noiseAndRate}.

Most of the literature discusses rate-induced tipping in the context of flows. In \cite{sebas}, R-tipping is studied in flows where parameters change according to a {\em parameter shift} (essentially, a smooth transition from one constant value to another). Conditions are given there and further developed in \cite{kiers} for when one might expect to see R-tipping in such a system. The goal of this paper is to explore the possibility of R-tipping in maps and to propose conditions for tipping in this kind of system as an analog to \cite{sebas} and \cite{kiers}. Some initial work was done on this in Chapter 5 of \cite{hahn}, and we hope to expand on their results.

Suppose we have a map of the form
\begin{equation}\label{map}
x_{n+1} = f(x_n,\lambda),
\end{equation}
where $x_i \in \R^\ell$, $\lambda \in \R^m$, and $f: \R^\ell \times \R^m \to \R^\ell$ is $C^1$. Note that $f( \cdot , \lambda)$ need not be invertible. We want to allow the parameter $\lambda$ to vary in a continuous way from one value to another over time, so we replace it with a $C^1$ function $\Lambda: \R \to \R^m$ satisfying
\begin{equation}
\begin{split}\label{paramShift}
\lim_{s \to \pm \infty} \Lambda(s) & = \lambda_\pm \\
\lim_{s \to \pm \infty} \Lambda'(s) & = 0
\end{split}
\end{equation}
for some $\lambda_\pm \in \R^n$. (In \cite{sebas}, such a function is called a {\em parameter shift}.) To allow the parameter change to happen at different rates, we introduce the rate $r > 0$ and obtain the map
\begin{equation}\label{mapWithRate}
x_{n+1} = f(x_n,\Lambda(rn)).
\end{equation}
When $r$ is close to $0$, $\Lambda(rn)$ changes slowly as $n$ increases, but when $r$ is large, $\Lambda(rn)$ approximates a step function from $\lambda_-$ to $\lambda_+$. Solutions to \cref{mapWithRate} also satisfy the map
\begin{equation}\label{autonomousMap}
\begin{split}
s_{n+1} &= s_n + r \\
x_{n+1} &= f(x_n,\Lambda(s_n))
\end{split}
\end{equation}
where $s_n = rn$. However, \cref{mapWithRate} and \cref{autonomousMap} are not equivalent because a solution $\{x_n, s_n\}$ to \cref{autonomousMap} does not have to satisfy $s_0 = 0$. We will refer to \cref{map} as the {\em autonomous map} and \eqref{mapWithRate} as the {\em nonautonomous map}. Notice that if $r = 0$, then the nonautonomous map reduces to the autonomous map where $\lambda = \Lambda(0)$. \par

For a square matrix $M$, let $\rho(M) = \max \{|\lambda|: \lambda \text{ is an eigenvalue of }M\}$ denote the spectral radius. Then we define a {\em path} as follows:
\begin{defn}
Suppose that for all $s \in \R$, $X(s)$ is an equilibrium of \eqref{map} with $\lambda = \Lambda(s)$ such that $(s,X(s))$ is a connected curve. Suppose also that this extends to the limits as $s \to \pm \infty$, so $X_\pm = \lim_{s \to \pm \infty} X(s)$ are equilibria of \eqref{map} with $\lambda = \lambda_\pm$, respectively. Then $(s,X(s))$ is
\begin{enumerate}
\item a {\em stable path} if $\rho(D_x f(X(s),\Lambda(s))) < 1$ for all $s \in \R \cup \{\pm \infty\}$;
\item an {\em unstable path} if $\rho(D_x f(X(s),\Lambda(s))) > 1$ for all $s \in \R \cup \{\pm \infty\}$.
\end{enumerate}
\end{defn}

Paths are not solutions to the nonautonomous map \cref{mapWithRate}; rather, they are a guideline against which we can compare solutions of the map. Note that paths are continuous curves, while solutions $\{x_n\}_{n \in \Z}$ of \cref{mapWithRate} are sequences of discrete points. In this paper, it will be helpful to plot paths $\{(s,X(s))\}_{s \in \R}$ and solutions $\{(rn,x_n)\}_{n \in \Z}$ together on the same axes; see \cref{pathsAndSolutions} as an illustration.

\begin{figure}
\centering
\includegraphics[scale = 0.55]{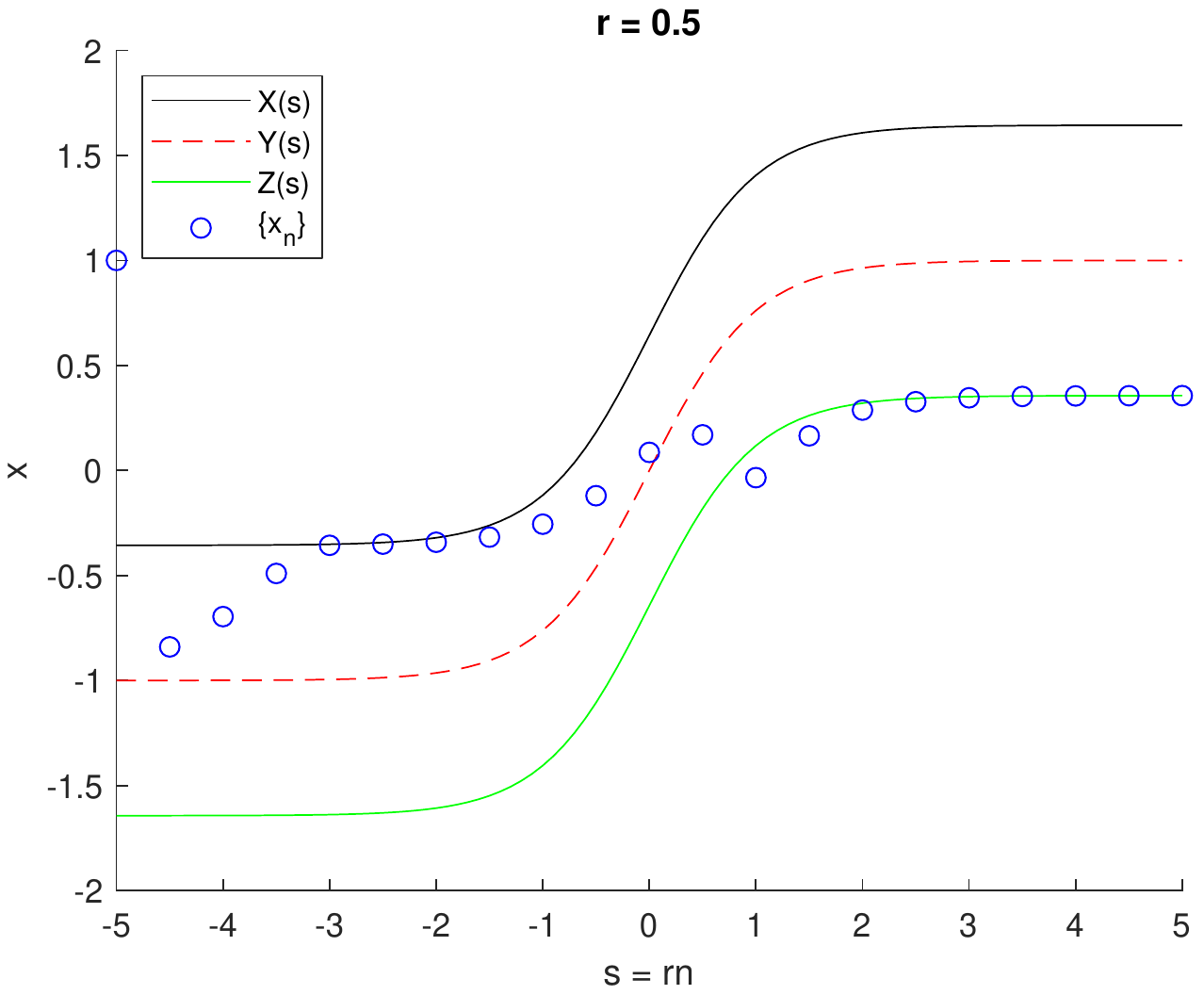}
\caption{Paths and a solution to \eqref{mapWithRate}. Both $(s,X(s)$ and $(s,Z(s))$ are stable paths, while $(s,Y(s))$ is an unstable path. The solution $\{x_n\}$ to \eqref{mapWithRate} with $r = 0.5$ and initial condition $x_{-10} = 1$ is plotted with blue circles where $s = 0.5n$.}
\label{pathsAndSolutions}
\end{figure}

In \cref{sec:UniquePullbackAttractor} we prove that for every stable path $(s,X(s))$ there is a unique solution $\{x_n^r\}$ of \eqref{mapWithRate} that approaches $X_-$ as $n \to -\infty$ and that this unique solution is a local pullback attractor. In \cref{fenichel}, we use geometric singular perturbation theory to show that if $r$ is sufficiently small, the pullback attractor to $X_-$ will approach $X_+$ as $n \to \infty$, which we call {\em endpoint tracking}.

In \cref{conditions} we define rate-induced tipping for maps and give some conditions under which one can expect R-tipping to happen or not. Some of these conditions are different from conditions for R-tipping in flows, and we highlight these differences in \cref{flowsMaps}, giving an example in which the continuous and discrete pullback attractors have drastically different behavior. Finally, in \cref{section:ikeda} we look at an example of rate-induced tipping in the 2-dimensional Ikeda map to highlight the fact that the results given in this paper are not restricted to maps of one dimension.

\section{Existence and Uniqueness of Local Pullback Attractors Beginning on Stable Paths}
\label{sec:UniquePullbackAttractor}

Assume that \cref{mapWithRate} has a stable path $(s,X(s))$ and $X_\pm = \lim_{s \to \pm \infty} X(s)$. The goal of this section is to prove that there is a unique solution of \cref{mapWithRate} that limits to $X_-$ as $n \to -\infty$ (\cref{uniqueSolution}) and that this solution is a local pullback attractor. The proof of \cref{uniqueSolution} relies on an important lemma, which we state and prove here first.

\begin{lemma}\label{epsilonNeigh}
For every sufficiently small $\epsilon > 0$ there exists an $S > 0$ such that if $rN>S$, then there is a unique solution $\{x_n\}$ of \eqref{mapWithRate} that stays within an $\epsilon$-neighborhood of $X_-$ for all $n \le -N$.
\end{lemma}
\begin{proof}
Let $\epsilon > 0$. For the sake of simplicity, suppose $X_- = 0$ and $\lambda_- = 0$. By definition of a stable path, all eigenvalues of $A := D_xf(0,0)$ have norm less than 1. If we pick $c \in (\rho(A),1)$, then by Lemma 5.6.10 of \cite{horn}, there is a matrix norm $||| \cdot |||$ such that $|||A||| < c$. Moreover, this is a matrix norm induced by a vector norm $|| \cdot ||$ on $\R^\ell$ (see Example 5.6.4 of \cite{horn}). Therefore,
\begin{equation}\label{norm}
||Ax_0|| \le |||A||| \ ||x_0|| < c ||x_0||
\end{equation}
for all $x_0 \in \R^\ell$. Since $f$ is $C^1$, for any $x_0,y_0 \in \R^\ell$ we can write
\begin{equation}\label{meanValue}
f(x_0,\lambda)-f(y_0,\lambda) = \left[\int_0^1 D_x f(tx_0 + (1-t)y_0,\lambda) dt\right] (x_0-y_0) =: g(x_0,y_0,\lambda)(x_0-y_0)
\end{equation}
where $g$ is continuous and $g(0,0,0) = A$. Therefore, by \cref{norm} and \cref{meanValue} we can make $\epsilon$ smaller if necessary and choose $\delta > 0$ such that if $x_0,y_0 \in \overline{B_\epsilon(0)}$ and $\lambda \in B_\delta(0)$, then
\begin{equation}\label{contraction}
|| f(x_0,\lambda) - f(y_0,\lambda) ||= ||g(x_0,y_0,\lambda)(x_0-y_0)|| < c ||x_0-y_0||.
\end{equation}

Now, let $S>0$ be sufficiently large so that $\Lambda(rn) \in B_\delta(0)$ and $||f(0,\Lambda(rn))|| \le \epsilon(1-c)$ for all $rn < -S$. Fix $r,N>0$ such that $rN > S$. Let $\mathcal{P}$ denote the set of sequences $\{x_n\}_{n = -\infty}^{-N}$ such that $x_n \in \overline{B_\epsilon(0)}$ for all $n \le -N$. For $\{x_n\} \in \mathcal{P}$, define
$$||\{x_n\}|| := \sup_{n \le -N} ||x_n||$$
using the same vector norm as above. We define the distance between two sequences in $\mathcal{P}$ to be $d(\{x_n\},\{y_n\}) := ||\{x_n - y_n\}||$. Then $\mathcal{P}$ is a complete metric space under this distance function. Define $\phi: \mathcal{P} \to \mathcal{P}$ by
\begin{align*}
&\phi(\ldots, x_{-N-2}, x_{-N-1}, x_{-N}) \\
&= \big(\ldots, f(x_{-N-3},\Lambda(r(-N-3))), f(x_{-N-2},\Lambda(r(-N-2))), f(x_{-N-1},\Lambda(r(-N-1)))\big).
\end{align*}
To verify that the image of $\mathcal{P}$ under $\phi$ is a subset of $\mathcal{P}$, pick $\{x_n\} \in \mathcal{P}$. For any fixed $n$, $x_n \in \overline{B_\epsilon(0)}$. Then
\begin{align*}
||f(x_n,\Lambda(rn))|| & \le ||f(x_n,\Lambda(rn)) - f(0,\Lambda(rn))|| + ||f(0,\Lambda(rn))|| \\
& < c||x_n|| + \epsilon(1-c) \text{, by \cref{contraction}}\\
& \le \epsilon c + \epsilon(1-c) \\
& = \epsilon,
\end{align*}
so $f(x_n,\Lambda(rn)) \in \overline{B_\epsilon(0)}$. Thus, $\phi(\{x_n\}) \in \mathcal{P}$. A sequence in $\mathcal{P}$ is a fixed point under $\phi$ if and only if it is the beginning of a solution of \cref{mapWithRate}. Furthermore, $\phi$ is a contraction mapping because if $\{x_n\},\{y_n\} \in \mathcal{P}$,
\begin{align*}
d(\phi(\{x_n\}),\phi(\{y_n\})) & = ||\phi(\{x_n\}) - \phi(\{y_n\}) || \\
& = \sup_{n \le -N} ||f(x_{n-1},\Lambda(r(n-1))) - f(y_{n-1},\Lambda(r(n-1)))|| \\
& \le c \sup_{n \le -N} ||x_{n-1} - y_{n-1}|| \text{, by } \cref{contraction} \\
& \le c \sup_{n \le -N} ||x_n - y_n|| \\
& = c \ d(\{x_n\},\{y_n\}).
\end{align*}
The unique fixed point of this mapping gives the only solution $\{x_n\}$ of \cref{mapWithRate} that stays within an $\epsilon$-neighborhood of $0$ for all $n \le -N$.
\end{proof}

Therefore, we can conclude

\begin{thm}\label{uniqueSolution}
For each $r > 0$, there is a unique solution $\{x_n^r\}$ of \cref{mapWithRate} satisfying $\lim_{n \to -\infty} x_n^r = X_-$.
\end{thm}
\begin{proof}
Fix $r > 0$. Pick $\epsilon > 0$ sufficiently small for \cref{epsilonNeigh}. Then there is an $N \in \N$ and a unique solution $\{x_n^r\}$ of \cref{mapWithRate} such that $||x_n^r - X_-|| < \epsilon$ for all $n \le -N$. Let $\delta \in (0,\epsilon)$. Then there is an $N_1 \ge N$ and a unique solution $\{y_n^r\}$ of \cref{mapWithRate} such that $||y_n^r - X_-|| < \delta$ for all $n \le -N_1$. Since $||x_n^r - X_-|| , ||y_n^r - X_-|| < \epsilon$ for all $n \le -N_1$, uniqueness implies $x_n^r = y_n^r$ for all $n \in \Z$. In particular, $||x_n^r - X_-|| < \delta$ for all $n \le -N_1$. Since our choice of $\delta \in (0,\epsilon)$ was arbitrary, $\lim_{n \to -\infty} x_n^r = X_-$.

Furthermore, there cannot be another $\{z_n^r\}$ satisfying $\lim_{n \to -\infty} z_n^r = X_-$ because that would violate uniqueness in \cref{epsilonNeigh}.
\end{proof}

Following the example of \cite{sebas}, we would like to use the term {\em local pullback attractor} to describe this unique solution $\{x_n^r\}$ from \cref{uniqueSolution}. To justify doing this, we will show that $\{x_n^r\}$ is related to the pullback attractors defined in \cite{nonauto}, although it is not compact and it may not be a global attractor. We introduce the following notation to help state precisely what we mean by local pullback attractor:

\begin{defn} Let $\phi(n_1,n_0,y)$ denote $y_{n_1}$ where $\{y_n\}$ is a solution to \cref{mapWithRate} with initial condition $y_{n_0} = y$ and $n_0 \le n_1$.
\end{defn}

Then we have

\begin{prop}\label{pullbackAttr}
The solution $\{x_n^r\}$ guaranteed by \cref{uniqueSolution} is a local pullback attractor, in the sense that there exists an open $U \subset \R^\ell$ containing $X_-$ such that for all $y \in U$,
$$\lim_{n_0 \to -\infty} ||\phi(n_1,n_0,y) - x_{n_1}^r|| = 0.$$
\end{prop}
We postpone the proof of \cref{pullbackAttr} until the end of \cref{fenichel}, since some notation and results given in \cref{fenichel} will be used in the proof.

Based on \cref{pullbackAttr}, it is appropriate to refer to $\{x_n^r\}$ as the {\em local pullback attractor to $X_-$} or simply the {\em pullback attractor to $X_-$}. If it's clear which backward limit point is being referred to from context, we may not mention $X_-$.

\section{Endpoint Tracking for Small Rates}\label{fenichel}
As established in \cref{sec:UniquePullbackAttractor}, if \cref{mapWithRate} has a stable path $(s,X(s))$ with $X_\pm = \lim_{s \to \pm \infty} X(s)$, the pullback attractor $\{x_n^r\}$ is the unique solution of \cref{mapWithRate} that approaches $X_-$ as $n \to -\infty$. When we eventually define rate-induced tipping, we will be looking at the behavior of $\{x_n^r\}$ as $n \to \infty$. In particular, we want to know whether $\lim_{n \to \infty} x_n^r = X_+$ or not:

\begin{defn}
Let $\{x_n^r\}$ be the pullback attractor to $X_-$. If $\lim_{n \to \infty} x_n^r = X_+$, then the pullback attractor {\em endpoint tracks} the path $(s,X(s))$.
\end{defn}

Our goal in this section is to establish that the pullback attractor will endpoint track the path as long as $r > 0$ is sufficiently small; in particular, we can force the pullback attractor to stay as close to the path as desired by choosing $r$ small. This result is similar to Fenichel's Theorem for the perturbation of invariant manifolds (see Theorem 1 of \cite{fenichel}). Unfortunately, we cannot rely on the classic result for what we want to show here for two reasons: we are working with a map (which is not necessarily a diffeomorphism) and the stable path is not compact. Nevertheless, we can prove our result using similar methods to the proof of Fenichel's Theorem (for example, see Chapter 2 of \cite{kuehn}). \par

To be precise, we want to prove

\begin{thm}\label{TheTheorem}
Let $\epsilon > 0$. When $r > 0$ is sufficiently small, there is a solution $\{y_n\}$ of \cref{mapWithRate} such that $||y_n - X(rn)|| < \epsilon$ for all $n \in \Z$. 
\end{thm}

As long as $\epsilon$ is chosen small enough, this orbit $\{y_n\}$ must be the pullback attractor to $X_-$ since it is the only orbit that stays within an $\epsilon$-neighborhood of $X_-$ as $n \to -\infty$. To prove \cref{TheTheorem}, it will be easier to work with equation \cref{autonomousMap}. We will find a solution to \cref{autonomousMap} satisfying the conditions of \cref{TheTheorem} and show that it corresponds to a solution of \cref{mapWithRate}.

Our plan for the proof of \cref{TheTheorem} is as follows: In \cref{uniformity,N,contractsOneThird,estimates,fPertContracts,oneHalf} we prove some necessary bounds on the functions of interest. Then we introduce the graph transform map, which acts on curves close to the stable path $(s,X(s))$. We prove that the graph transform is a contraction mapping and hence that there is a unique curve close to $(s,X(s))$ which is invariant under \cref{autonomousMap}. From this curve we will be able to construct the solution $\{y_n\}$ mentioned in \cref{TheTheorem}.

We begin with some setup and notation. Let $M = \{(s,X(s)) : s \in \R\}$ denote the entire stable path. Define $F_r: \R^{\ell + 1} \to \R^{\ell + 1}$ by
$$F_r(s,x) = (s + r,f(x,\Lambda(s)))$$
so that \cref{autonomousMap} can be written as $(s_{n+1},x_{n+1}) = F_r(s_n,x_n)$. Notice that $M$ is invariant under the map $F_0$ (when $r = 0$); see \cref{critMan}. We will think of $F_r$ (for $r > 0$) as a perturbation of $F_0$ and look for an invariant manifold under $F_r$ that is close to $M$. This manifold will give us a solution to \cref{autonomousMap} close to the stable path.

\begin{figure}
	\centering
	\begin{subfigure}{0.48 \textwidth}
		\centering
		\includegraphics[width = \textwidth]{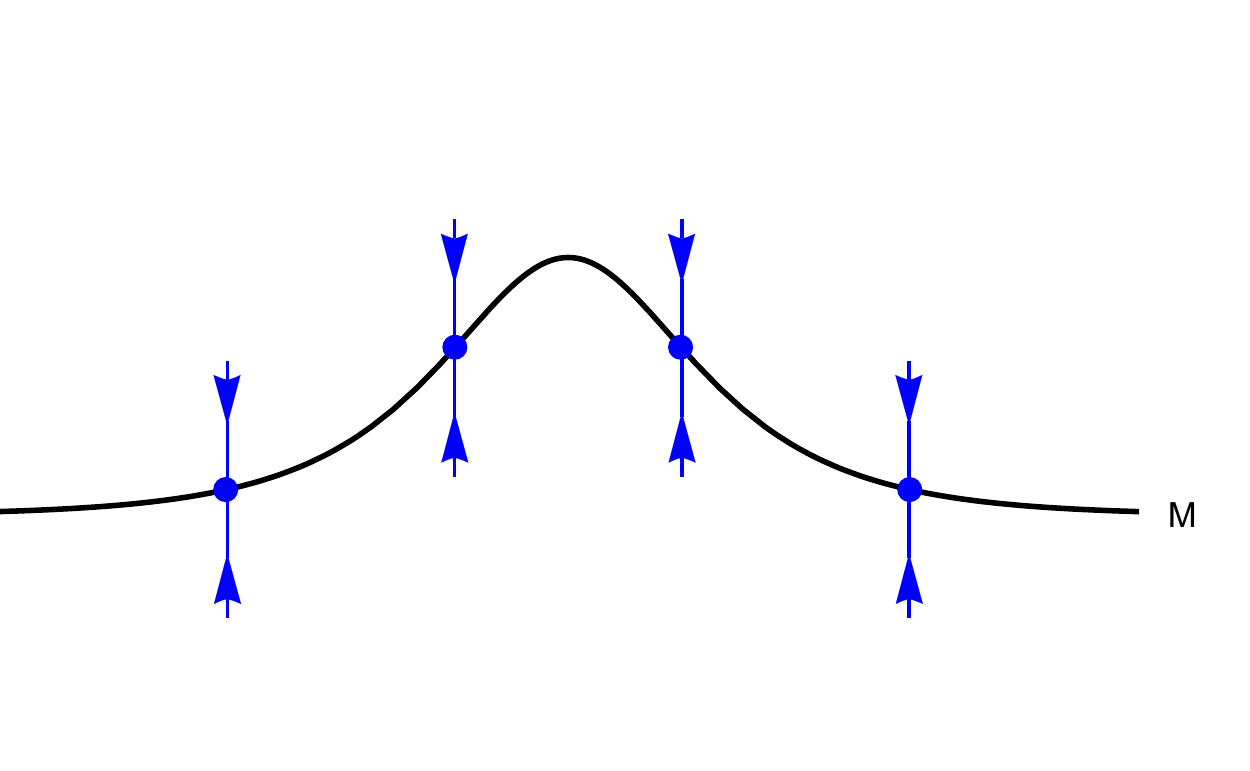}
		\caption{Every point of $M$ is an attracting fixed point under $F_0$. Sets of the form $\{s = s_0\}$ are invariant under $F_0$ since the $s$-coordinate does not change.}
		\label{critMan}
	\end{subfigure}
	\hfill
	\begin{subfigure}{0.48 \textwidth}
		\centering
		\includegraphics[width = \textwidth]{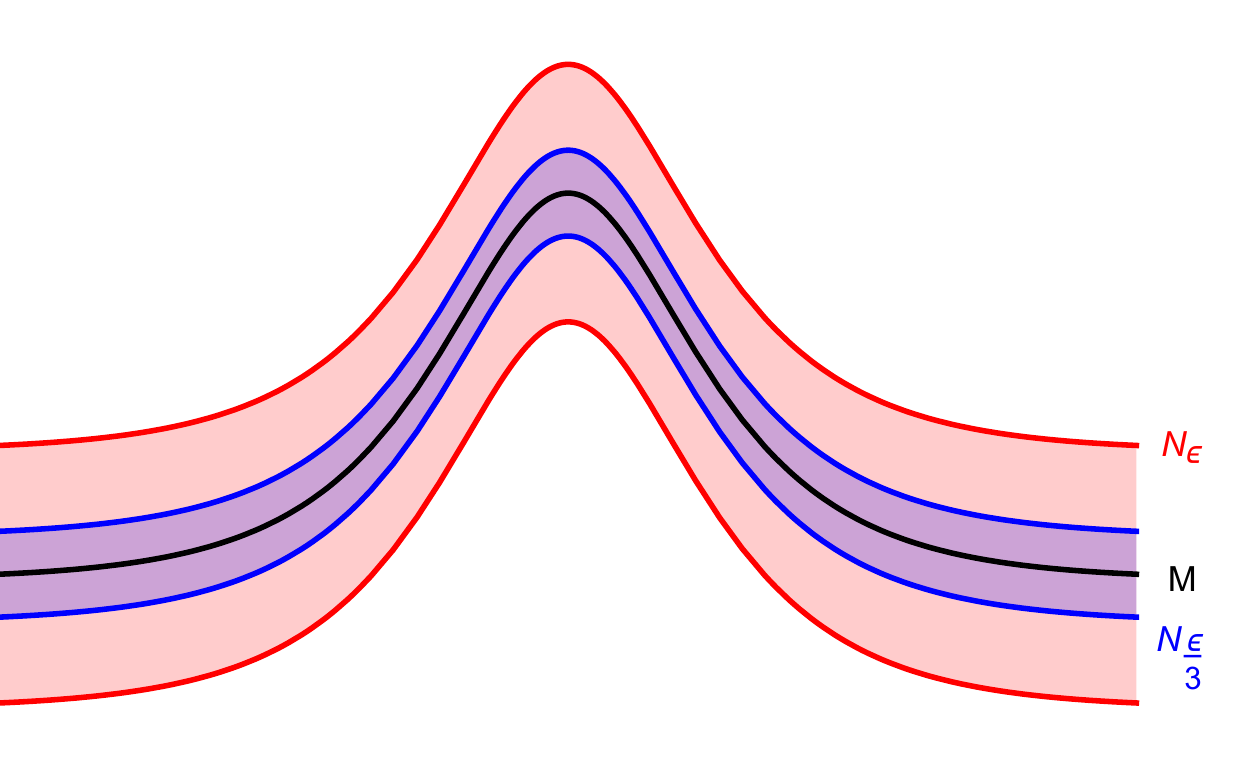}
		\caption{Neighborhoods $\mathcal{N}_\epsilon$ and $\mathcal{N}_{\epsilon/3}$ around $M$; \cref{contractsOneThird} shows that $F_0^n: \mathcal{N}_\epsilon \to \mathcal{N}_{\epsilon/3}$ for certain $n$.}
		\label{NEpsilon}
	\end{subfigure}
	\caption{}
	\label{}
\end{figure}

For notational convenience, define
$$B_n(s) = [D_x f(X(s),\Lambda(s))]^n = D_x f^n(X(s),\Lambda(s)).$$

Let $|| \cdot ||$ denote the standard Euclidean norm as well as the induced matrix norm on square matrices. Then we have the following lemma, similar to the Uniformity Lemma of \cite{kuehn}:

\begin{lemma}\label{uniformity}
There are constants $\kappa > 0$ and $c \in (0,1)$ such that
$$||B_n(s)|| < \kappa c^n$$
for every $s \in \R$ and $n \in \N$.
\end{lemma}
\begin{proof}
Since $(s,X(s))$ is a stable path, there is some $c \in (0,1)$ satisfying $\rho(D_x f(X(s),\Lambda(s))) < c$ for all $s \in \R \cup \{\pm \infty\}$ (where $\rho$ denotes the spectral radius). By Gelfand's formula (Corollary 5.6.14 of \cite{horn}),
$$\lim_{n \to \infty} ||B_n(s)||^{1/n} = \rho(D_x f(X(s),\Lambda(s)) < c.$$
We do not know a priori that this limit is uniform in $s$; this is what we want to show.

For each $s$, there exists an $N(s) \in \N$ such that $||B_{N(s)}(s)|| < c^{N(s)}$. By the continuity of $B_{N(s)}$, there is a neighborhood $U(s)$ of $s$ such that for all $s' \in U(s)$,
$$||B_{N(s)}(s')|| < c^{N(s)}.$$
This is also true for $s = \pm \infty$, so we can get a finite subcovering $\{U(s_i)\}_{i = 1}^K$ for the extended real line. Set
$$\kappa = \sup_{s \in \R, \ r \in [0,\max N(s_i)]} \frac{||B_r(s)||}{c^r}.$$
Note that the supremum must be finite because $B_r(s)$ approaches constant matrices as $s \to \pm \infty$. \par

Pick $s \in \R$ and find $s_i$ such that $s \in U(s_i)$. Then pick any $n \in \N$ and write $n = d N(s_i) + r$, where $d,r \in \Z$ and $0 \le r < N(s_i)$. Then
$$B_n(s) = B_{dN(s_i) + r}(s) = [B_{N(s_i)}(s)]^d B_r(s).$$

If we take the norm, we get
\begin{align*}
||B_n(s)|| & \le ||B_{N(s_i)}(s)||^d ||B_r(s)|| \\
& < c^{dN(s_i)} ||B_r(s)|| \\
& = c^n \frac{||B_r(s)||}{c^r} \\
& \le \kappa c^n.
\end{align*}
\end{proof}

We can immediately conclude

\begin{cor}\label{N}
There exists an $N \in \N$ such that $||B_n(s)|| < \frac{1}{4}$ for all $s \in \R$ and $n \ge N$.
\end{cor}

Define a small closed neighborhood around $M$ by
$$\mathcal{N}_\epsilon = \{(s,x): s \in \R, ||x - X(s)|| \le \epsilon\}.$$
See \cref{NEpsilon} for an illustration.

Our next lemma shows that $F_0^n$ ($F_0$ composed with itself $n$ times) contracts $\mathcal{N}_\epsilon$ toward $M$ by at least $\frac{1}{3}$ for certain $n$.

\begin{lemma}\label{contractsOneThird}
For the $N$ in \cref{N} and for $\epsilon > 0$ small enough, $F_0^n: \mathcal{N}_\epsilon \to \mathcal{N}_{\epsilon/3}$ for all $n \in \{N, N+1, \ldots, N(N+1)\}$.
\end{lemma}
\begin{proof}
\cref{N} implies that the linear part of $F_0^n$ contracts $\mathcal{N}_\epsilon$ toward $M$ by at least $\frac{1}{4}$. The nonlinear part of $F_0^n$ is $\mathcal{O}(\epsilon^2)$ and approaches a constant function in the limits as $s \to \pm \infty$, so if we take $\epsilon$ small enough, we will get that $F_0^n: \mathcal{N}_\epsilon \to \mathcal{N}_{\epsilon/3}$ for all $n \in \{N, N+1, \ldots, N(N+1)\}$.
\end{proof}

Define $\pi: \R^{\ell + 1} \to \R^\ell$ to be the projection onto the last $\ell$ coordinates, so $\pi(s,x) = x$. Then $\pi \circ F_r^n: \R^{\ell + 1} \to \R^\ell$, and if $D_x$ denotes the derivative with respect to the last $\ell$ coordinates, then $D_x [\pi \circ F_r^n]$ is an $\ell \times \ell$ matrix.

The next lemma gives an estimate of how close together $F_0^n$ and $F_r^n$ (as well as their derivatives) must be in terms of $r$ and $n$.

\begin{lemma}\label{estimates}
Let $\mathcal{K}$ be a neighborhood of $M$ that is uniformly bounded in the $x$-directions. Then there exist constants $\kappa_n$ depending only on $F_0$ and $n$ such that
$$||F_0^n(s,x) - F_r^n(s,x)||, ||D_x[\pi \circ F_0^n](s,x) - D_x[\pi \circ F_r^n](s,x)|| \le r \kappa_n,$$
for all $(s,x) \in \mathcal{K}$ as long as $F_0^i(s,x)$ and $F_r^i(s,x)$ remain in $\mathcal{K}$ for all $i \le n-1$.
\end{lemma}
\begin{proof}
We will prove the result for $F_0, \ F_r$; the result for the derivatives can be proven similarly.

Let $\eta$ be a Lipschitz constant for $F_0$ on $\mathcal{K}$ (which exists because $F_0$ is $C^1$, $\mathcal{K}$ is uniformly bounded in the $x$-directions, and $F_0$ approaches $(s,x) \mapsto (s, f(x,\lambda_\pm))$ as $s \to \pm \infty$). \par

We will use induction to show that $||F_0^n(s,x) - F_r^n(s,x)|| \le r \sum_{i = 0}^{n-1} \eta^i$ for $n \in \N$. The base case when $n = 1$ is true because
$$||F_0(s,x) - F_r(s,x)|| = ||(s,f(x,\Lambda(s))) - (s+r,f(x,\Lambda(s)))|| = ||(-r,0)|| = r.$$

Then, if we assume the statement is true for a given $n \ge 1$,
\begin{align*}
||F_0^{n+1}(s,x) - F_r^{n+1}(s,x)|| & \le ||F_0^{n+1}(s,x) - F_0(F_r^n(s,x))|| + ||F_0(F_r^n(s,x)) - F_r^{n+1}(s,x)|| \\
& \le \eta||F_0^n(s,x) - F_r^n(s,x)|| + r \\
& \le r \eta\left(\sum_{i = 0}^{n-1} \eta^i\right) +r \\
& = r \sum_{i = 0}^n \eta^i
\end{align*}
\end{proof}

Now we can apply \cref{contractsOneThird,estimates} to show that $F_r^n$ maps $\mathcal{N}_\epsilon$ into itself for certain $n$.

\begin{lemma}\label{fPertContracts}
For $r$ and $\epsilon$ sufficiently small, we have
$$F_r^n\left( \mathcal{N}_\epsilon \right) \subset \mathcal{N}_\epsilon$$
for any $n \in \{N,N+1,\ldots,N(N+1)\}$.
\end{lemma}
\begin{proof}
By \cref{contractsOneThird}, $F_0^n: \mathcal{N}_\epsilon \to \mathcal{N}_{\epsilon/3}$ for all $n \in \{N, N+1, \ldots, N(N+1)\}$ if $\epsilon > 0$ is small enough. Pick some set $\mathcal{K}$ that is uniformly bounded in the $x$-directions such that $\mathcal{K}$ contains an open neighborhood of $F_0^n(\mathcal{N}_\epsilon)$ for all $n \in \{0,1,\ldots,N(N+1)\}$. Let $\kappa_n$ be as in \cref{estimates}. Pick $\delta > 0$ so that if $(s,x) \in \mathcal{N}_{\epsilon/3}$, then the $\delta$-ball around $(s,x)$ is contained in $\mathcal{N}_\epsilon$. Make $r > 0$ as small as needed so that $r \kappa_n < \delta$ and $F_r^n(\mathcal{N}_\epsilon) \subset \mathcal{K}$ for all $n \in \{0,1,\ldots,N(N+1)\}$.

Pick some $(s,x) \in \mathcal{N}_\epsilon$. Then $F_0^n(s,x) \in \mathcal{N}_{\epsilon/3}$ for all $n \in \{0,1,\ldots,N(N+1)\}$. By \cref{estimates},
$$||F_0^n(s, x) - F_r^n(s, x)|| < r \kappa_n < \delta.$$
This implies that $F_r^n(s, x) \in \mathcal{N}_\epsilon$. Hence, $F_r^n(\mathcal{N}_\epsilon) \subset \mathcal{N}_\epsilon$.
\end{proof}

Finally, we combine previous results to put a bound on the derivative of $\pi \circ F_r^n$.

\begin{lemma}\label{oneHalf}
Assume that $r$ and $\epsilon$ are sufficiently small. Then for all $n \in \{N, N+1, \ldots, N(N+1)\}$ and $(s,x) \in \mathcal{N}_\epsilon$,
$$||D_x [\pi \circ F_r^n](s,x)|| < \frac{1}{2}.$$
\end{lemma}
\begin{proof}
By the triangle inequality, we have
$$||D_x [\pi \circ F_r^n](s,x)|| \le ||D_x [\pi \circ F_r^n](s,x) - D_x [\pi \circ F_0^n](s,x)|| + ||D_x [\pi \circ F_0^n](s,x) - B_n(s)|| + ||B_n(s)||.$$
By \cref{estimates}, we can make $||D_x[\pi \circ F_0^n](s,x) - D_x[\pi \circ F_r^n](s,x)|| < \frac{1}{8}$ on $\mathcal{N}_\epsilon$ for all $n \le N(N+1)$ if $r$ is sufficiently small, and $||D_x [\pi \circ F_0^n](s,x) - B_n(s)|| < \frac{1}{8}$ if $\epsilon$ is small for all $n \in \{N, N+1,\ldots, N(N+1)\}$. Finally, we know from \cref{N} that $||B_n(s)|| < \frac{1}{4}$ for all $s \in \R$ and $n \ge N$. Therefore, we get the desired inequality.
\end{proof}

Let $S$ denote the space of functions $u: \R \to \mathcal{N}_\epsilon$. It is a complete metric space under the sup-norm: $||u||_\infty = \sup_{s \in \R} ||u(s)||$. The graph of $u$ is $\graph(u) = \{(s,u(s)): s \in \R\}$. The map $F_r^n$ defines a map on $S$ in the following way:
$$(s,u(s)) \mapsto (s + nr, \pi \circ F_r^n(s,u(s))) =: (\xi, (G_n u)(\xi)),$$
where $G_n: S \to S$ is called the {\em graph transform} (see \cref{imageOfGraphTransform}). The graph transform is well-defined because $s \mapsto s +nr = \xi$ is a bijection on $\R$, and \cref{fPertContracts} implies that the image of $S$ under $G_n$ truly is a subset of $S$. A fixed point under $G_n$ satisfies
$$F_r^n(\graph(u)) = \graph(u).$$

\begin{figure}
\centering
\includegraphics[scale = 0.7]{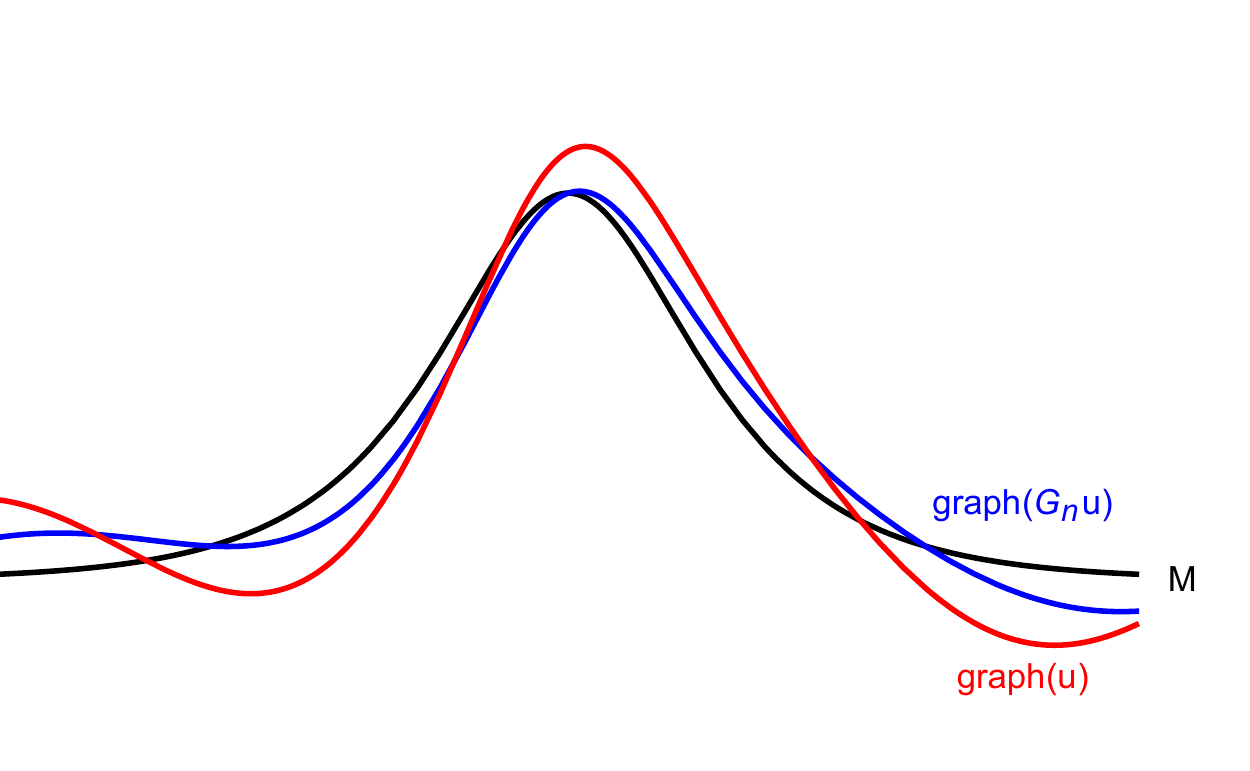}
\caption{The graph transform $G_n$ maps one function $u: \R \to \mathcal{N}_\epsilon$ to another. As is shown in \cref{contractionMapping}, $G_n$ is a contraction mapping, so repeated application of $G_n$ will transform any such function toward a function that is fixed under $G_n$.}
\label{imageOfGraphTransform}
\end{figure}

We will use the contraction mapping theorem to show that there is a unique fixed point of $G_n$ in $S$ for each $n \in \{N,N+1,\ldots,N(N+1)\}$.

\begin{lemma}\label{contractionMapping}
For $\epsilon$ and $r$ sufficiently small and $u,u' \in S$,
$$||G_nu - G_nu'||_\infty \le \frac{1}{2} ||u - u'||_\infty$$
for all $n \in \{N, N+1,\ldots, N(N+1)\}$.
\end{lemma}
\begin{proof}
Choose $s \in \R$. Then
\begin{align*}
||(G_nu)(s) - (G_nu')(s)|| & = ||\pi \circ F_r^n(s-nr,u(s-nr)) - \pi \circ F_r^n(s-nr,u'(s-nr))|| \\
& = ||D_x [\pi \circ F_r^n](s-nr,u^*)(u(s-nr) - u'(s-nr))|| \text{, for some } u^*\\
& < \frac{1}{2} ||u(s-nr) - u'(s-nr)|| \text{, by \cref{oneHalf}} \\
& \le \frac{1}{2} ||u-u'||_\infty
\end{align*}
Therefore,
$$||G_nu - G_nu'||_\infty \le \frac{1}{2} ||u - u'||_\infty.$$
\end{proof}

And now we are ready to prove \cref{TheTheorem}:

\begin{proof}[Proof of \cref{TheTheorem}]
By \cref{contractionMapping}, each $G_n$ for $n \in \{N, N+1,\ldots, N(N+1)\}$ is a contraction mapping of the complete metric space $S$. Therefore, by the contraction mapping theorem, there is a unique fixed point $u_n$ of $G_n$. This implies
$$F_r^n(\graph(u_n)) = \graph(u_n).$$
In particular,
\begin{align*}
F_r^N(\graph(u_N)) & = \graph(u_N) \\
F_r^{N+1}(\graph(u_{N+1})) & = \graph(u_{N+1})
\end{align*}
so
\begin{align*}
F_r^{N(N+1)}(\graph(u_N)) = \graph(u_N) \\
F_r^{N(N+1)}(\graph(u_{N+1})) = \graph(u_{N+1})
\end{align*}
By the uniqueness of the fixed point under $F_r^{N(N+1)}$, $\graph(u_N) = \graph(u_{N+1})$, so $u_N = u_{N+1}$. Let this common function be called $u$. Then
$$\graph(u) = F_r^{N+1}(\graph(u)) = F_r(F_r^N(\graph(u))) = F_r(\graph(u)).$$
This means that
$$F_r(s,u(s)) = (s+r, u(s+r))$$
for all $s \in \R$. For $k \in \Z$, set $y_k = u(rk)$. Then,
$$F_r(rk,y_k) = F_r(rk,u(rk)) = (r(k+1),u(r(k+1))) = (r(k+1),y_{k+1}).$$
This implies that $\{(rk,y_k)\}_{k \in \Z}$ is a solution of \cref{autonomousMap} and hence that $\{y_k\}_{k \in \Z}$ is a solution of \cref{mapWithRate}. Finally, we know
$$||y_k - X(rk)|| = ||u(rk) - X(rk)|| < \epsilon$$
since $u \in \mathcal{N}_\epsilon$.
\end{proof}

Based on the discussion immediately following the statement of \cref{TheTheorem}, we can phrase this result in terms of the pullback attractor to $X_-$.

\begin{cor}\label{epsilonDistance}
For any $\epsilon > 0$, there exists an $r_0 > 0$ such that if $r \in (0,r_0)$, then the pullback attractor $\{x_n^r\}$ to $X_-$ satisfies $||x_n^r - X(rn)|| < \epsilon$ for all $n \in \Z$.
\end{cor}

\newpage

\begin{ex}\end{ex}

As an example of \cref{epsilonDistance}, consider the 1-dimensional logistic map
$$x_{n+1} = f(x_n,\lambda) = \lambda x_n (1-x_n).$$
When $1 < \lambda < 3$, there is a unique stable fixed point at $x = 1 - \frac{1}{\lambda}$. By setting $\Lambda(rn) = 2 + .9 \tanh(rn)$ and plugging this in for $\lambda$ in the autonomous logistic map, we get a nonautonomous map of the form \cref{mapWithRate}. As $s$ varies over $\R$, $\Lambda(s)$ varies between $1.1$ and $2.9$, and $X(s) = 1 - \frac{1}{\Lambda(s)}$ is a stable path. \cref{LogisticMapFigs} demonstrates that the smaller $r > 0$ is, the closer the pullback attractor $\{x_n^r\}$ must be to the path.

\begin{figure}
	\centering
	\includegraphics[scale = 0.55]{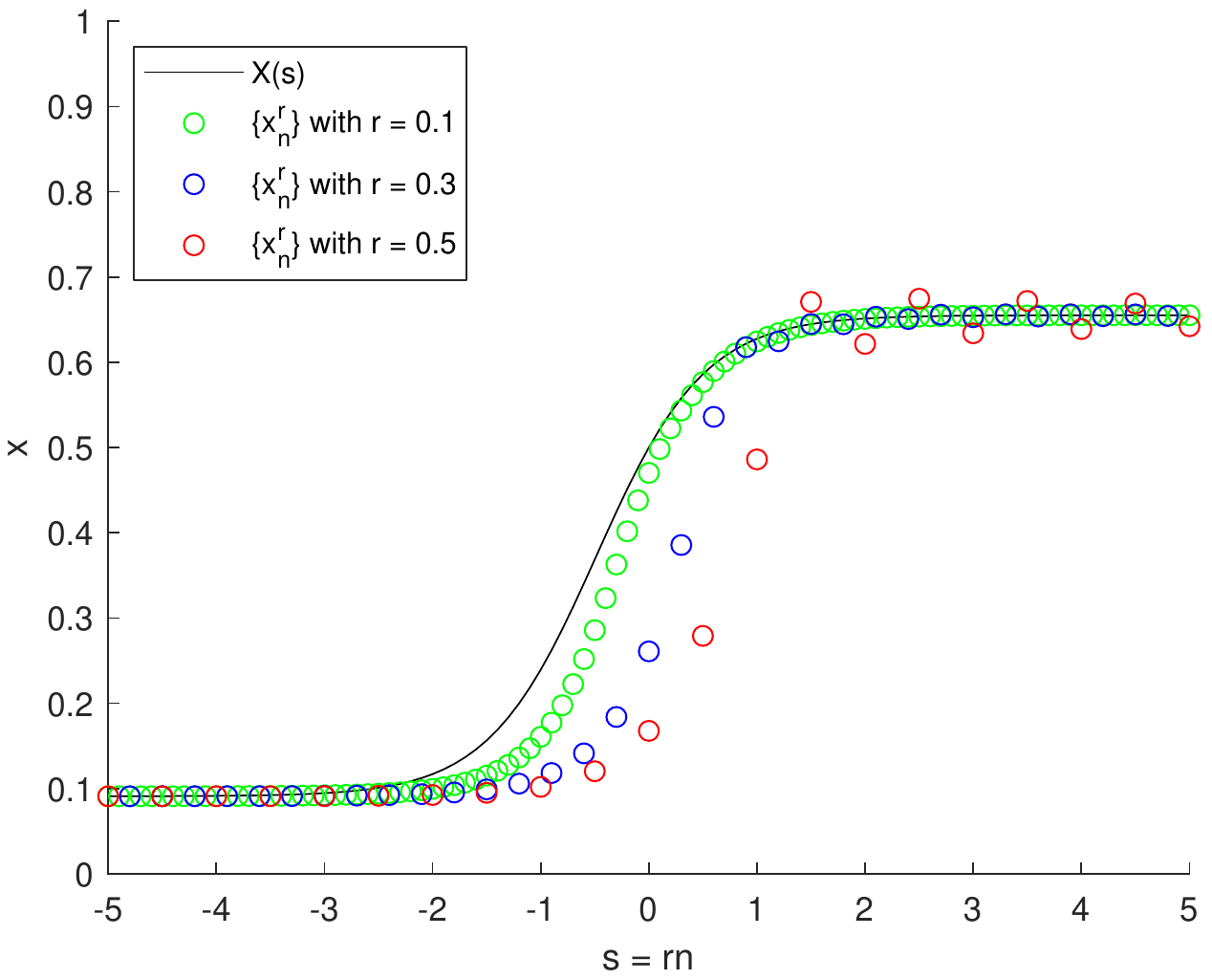}
	\caption{Three different pullback attractors to $X_-$ are plotted, with differing values of $r$. When $r$ is smaller, the pullback attractor stays closer to the path.}
	\label{LogisticMapFigs}
\end{figure}

We still have yet to show that for small $r > 0$ the pullback attractor will endpoint track the path $(s,X(s))$. Intuitively, we might expect that choosing $\epsilon > 0$ small enough and forcing the pullback attractor to stay within $\epsilon$ of the path would force endpoint tracking. That is indeed the case, and we prove it \cref{endpointTracks}. First, we need some notation and a lemma.

Let $\B(X(s),\Lambda(s))$ denote the basin of attraction for $X(s)$ in the autonomous map \cref{map} when $\lambda = \Lambda(s)$.

\begin{lemma}\label{compactK}
Let $K \subset \B(X_+,\lambda_+)$ be compact. Then there exists an $S > 0$ such that if $x_N \in K$ for $rN \ge S$, then $\lim_{n \to \infty} x_n = X_+$.
\end{lemma}
A nearly identical statement is made in Lemma 5 of \cite{kiers} for flows instead of maps, but the proof is the same. We omit the details here. Finally, we can conclude

\begin{thm}\label{endpointTracks}
There exists an $r_0 > 0$ such that if $r \in (0,r_0)$, then the pullback attractor $\{x_n^r\}$ to $X_-$ will endpoint track the path $(s,X(s))$.
\end{thm}
\begin{proof}
There is some $\epsilon > 0$ such that $\overline{B_\epsilon(X_+)} \subset \B(X_+,\lambda_+)$. By \cref{epsilonDistance}, there is an $r_0 > 0$ such that for all $r \in (0,r_0)$, $x_n^r \in \overline{B_\epsilon(X_+)}$ for sufficiently large $n$. By \cref{compactK}, if $x_n^r \in \overline{B_\epsilon(X_+)}$ for large $n$, then $\lim_{n \to \infty} x_n^r = X_+$. Therefore, if $r \in (0,r_0)$, $\lim_{n \to \infty} x_n^r = X_+$.
\end{proof}

Finally, we finish this section with a proof of \cref{pullbackAttr}, which states that what we have been calling the pullback attractor is indeed a local pullback attractor:

\begin{proof}[Proof of \cref{pullbackAttr}]
Take $U = B_\epsilon(X_-)$, where $\epsilon > 0$ is the same as in \cref{contractionMapping}. Let $y \in U$. Pick an element $u \in S$ satisfying $u(s) = y$ for all sufficiently small $s$. Let $v \in S$ denote the unique fixed point under the graph transform $G$ guaranteed by the proof of \cref{TheTheorem} so $v(n) = x_n$ for all $n \in \Z$. We know $G$ acts as a contraction mapping on $S$, so
$$||G_n u - v||_\infty \to 0$$
as $n \to \infty$. Fix any $n_1 \in \Z$, and let $\delta > 0$. Then there exists an $N \in \N$ such that for all $n \ge N$, $||G_n u - v||_\infty < \delta$. Making $N$ larger if necessary, we can also guarantee that $u(n_1 - n) = y$ for all $n \ge N$. Then in particular, 
$$||G_nu(n_1) - v(n_1)|| = ||\phi(n_1,n_1 - n, y) - x_{n_1}|| < \delta$$
for all $n \ge N$. Therefore, $\phi(n_1,n_0,y) \to x_{n_1}$ as $n_0 \to -\infty$.
\end{proof}

\section{Definition and Conditions for Rate-induced Tipping in Maps}\label{conditions}

Now that we have proven some preliminary results in \cref{sec:UniquePullbackAttractor,fenichel}, we are ready to define rate-induced tipping. As before, suppose \cref{mapWithRate} has a stable path $(s,X(s))$ with $X_\pm = \lim_{s \to \pm \infty} X(s)$. \cref{endpointTracks} guarantees that as long as $r > 0$ is sufficiently small, the pullback attractor $\{x_n^r\}$ to $X_-$ endpoint tracks the path; that is, that $\lim_{n \to \infty} x_n^r = X_+$. However, for some maps and paths there may be a rate $r_0 > 0$ for which the pullback attractor does {\em not} endpoint track the path, or $\lim_{n \to \infty} x_n^{r_0} \ne X_+$ or does not exist. This will be our definition for rate-induced tipping in a map.

\begin{defn}
If there is some $r_0 > 0$ for which the pullback attractor to $X_-$ does not endpoint track the path $(s,X(s))$, then {\em rate-induced tipping} has occurred.
\end{defn}

\begin{ex}\end{ex}

To demonstrate that rate-induced tipping can happen in maps, consider
\begin{align*}
f(x,\lambda) & = \frac{2(x-\lambda)}{(1+(x-\lambda)^2)^2} + \lambda \\
\Lambda(s) & = \tanh(s).
\end{align*}
Then there are two stable paths, $X(s) = \tanh(s) + \sqrt{-1 + \sqrt{2}}$ and $Z(s) = \tanh(s) - \sqrt{-1 + \sqrt{2}}$, as well as an unstable path, $Y(s) = \tanh(s)$. The limiting values are $X_\pm = \pm 1 + \sqrt{-1 + \sqrt{2}}$, $Y_\pm = \pm 1$, and $Z_\pm = \pm 1 - \sqrt{-1 + \sqrt{2}}$, corresponding to $\lambda_\pm = \pm 1$. In \cref{RTippingDemo} we plot these paths along with the pullback attractor $\{x_n^r\}$ to $X_-$ for two different values of $r>0$. When $r$ is small, the pullback attractor tracks $(s,X(s))$, but when $r$ is large, it does not, and there is rate-induced tipping.

\begin{figure}
	\centering
	\begin{subfigure}{0.4 \textwidth}
		\centering
		\includegraphics[width = \textwidth]{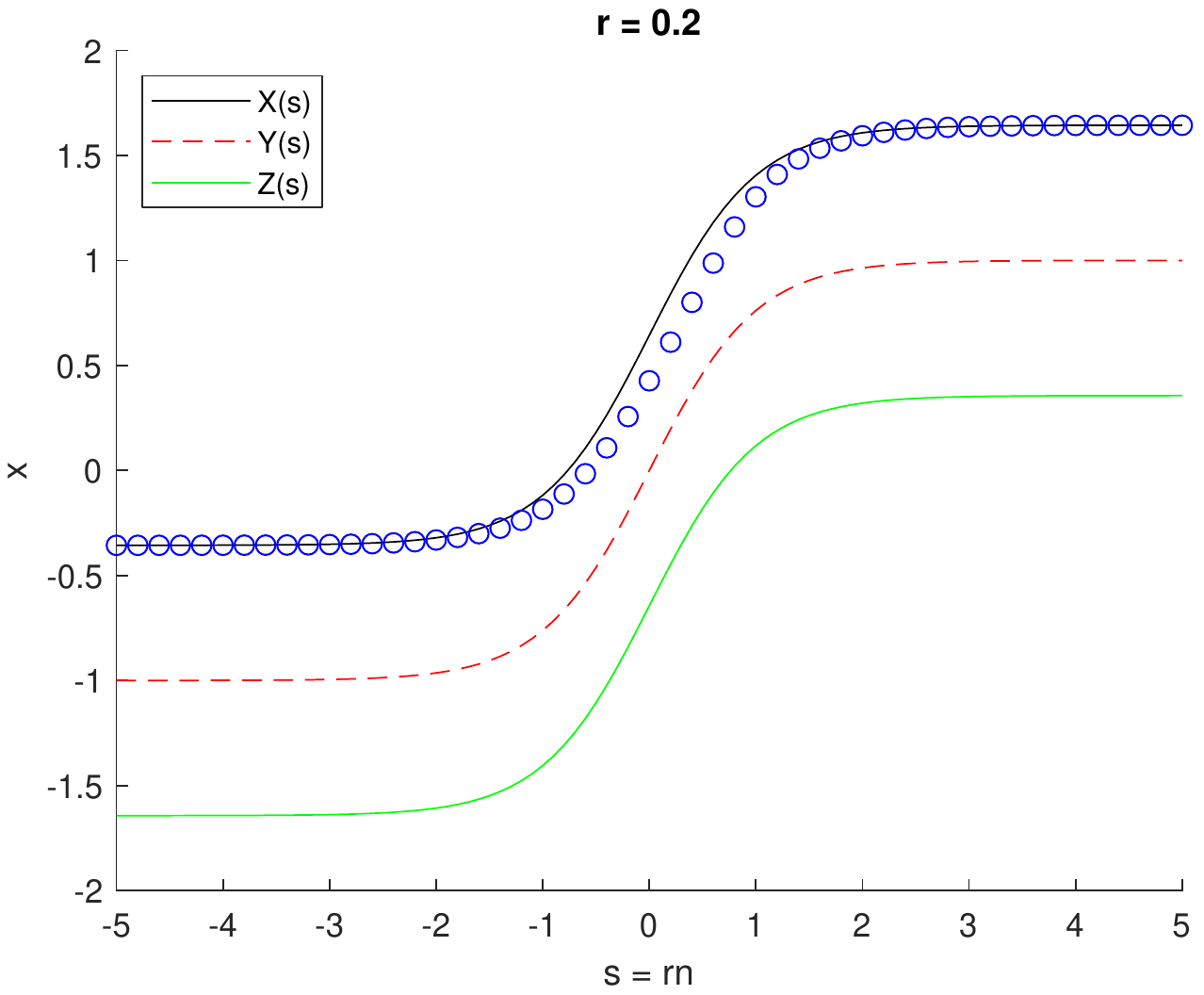}
	\end{subfigure}
	\hspace{0.5in}
	\begin{subfigure}{0.4 \textwidth}
		\centering
		\includegraphics[width = \textwidth]{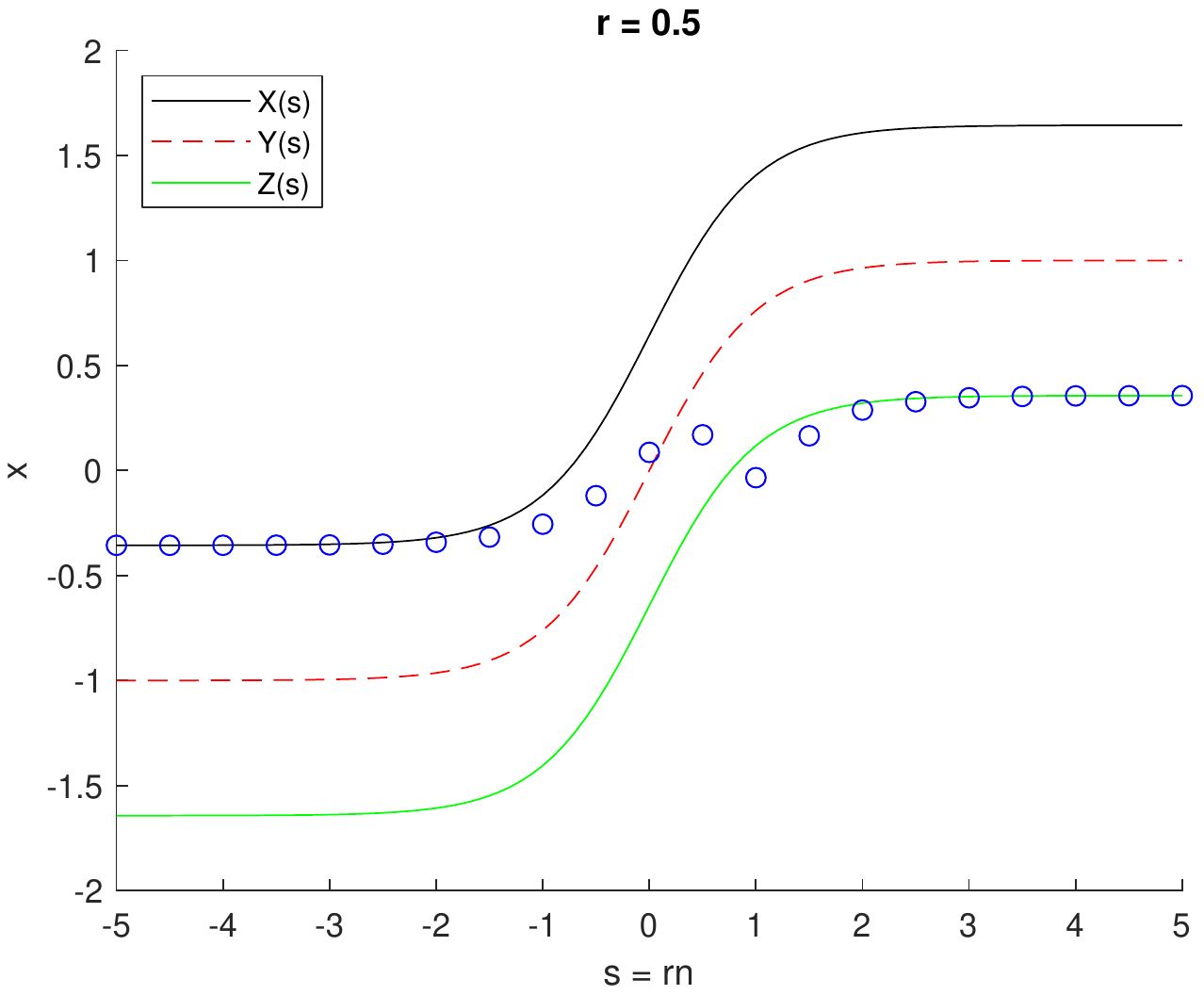}
	\end{subfigure}
	\caption{When $r = 0.2$, the pullback attractor to $X_-$ endpoint tracks the path $(s,X(s))$, but when $r = 0.5$, the pullback attractor does not endpoint track, showing that there is rate-induced tipping.}
	\label{RTippingDemo}
\end{figure}

The rest of this section will be devoted to exploring which kinds of parameter shifts lead to rate-induced tipping and which don't. The same thing is done for flows in \cite{sebas} and \cite{kiers}, and we will compare and contrast these results with R-tipping results for maps. As we will see, some of the R-tipping results are the same between continuous and discrete dynamical systems, but some are quite different.

We begin with the notion of {\em forward basin stability}, introduced in \cite{sebas}. We say a stable path $(s,X(s))$ is {\em forward basin stable} if
$$\overline{\{X(s_1): s_1 < s_2\}} \subset \B(X(s_2),\Lambda(s_2))$$
for all $s_2 \in \R$.

It is shown in \cite{sebas} that forward basin stability prevents R-tipping in 1-dimensional flows, but from \cite{kiers} we know it does not necessarily prevent R-tipping in flows of higher dimensions. As the following example shows, forward basin stability does not prevent R-tipping even in maps of one dimension.

\ignore{Intuitively, the reason for this is that in maps of any dimension (as in flows of dimension higher than 1), a point $x$ may be in the basin of attraction of some other point $p$, but that does not mean that $\phi(n,n_0,x)$ moves in the direction of $p$ or even gets strictly closer to $p$ as $n$ increases. We use this fact to our advantage to construct the following example.}

\begin{ex}\label{FBSnotEnough}\end{ex}

In the nonautonomous map
$$f(x,\lambda) = .25((x-\lambda)^2 - (x-\lambda)) + \lambda$$
$$\Lambda(rn) = 3 \sech((rn)^{10})$$
there is a unique stable path $(s,X(s))$ with $X(s) = \Lambda(s)$ and a unique unstable path $(s,Y(s))$ with $Y(s) = 5 + \Lambda(s)$. For any $s \in \R$, $\B(X(s),\Lambda(s)) = (-4 + \Lambda(s), Y(s))$. With the given parameter shift $\Lambda(s)$, $(s,X(s))$ is forward basin stable. However, the pullback attractor $\{x_n^r\}$ to $X_- = 0$ does not endpoint track the path for $r = 2$ (see \cref{fig:FBS}).

\begin{figure}
\centering
\includegraphics[scale = 0.55]{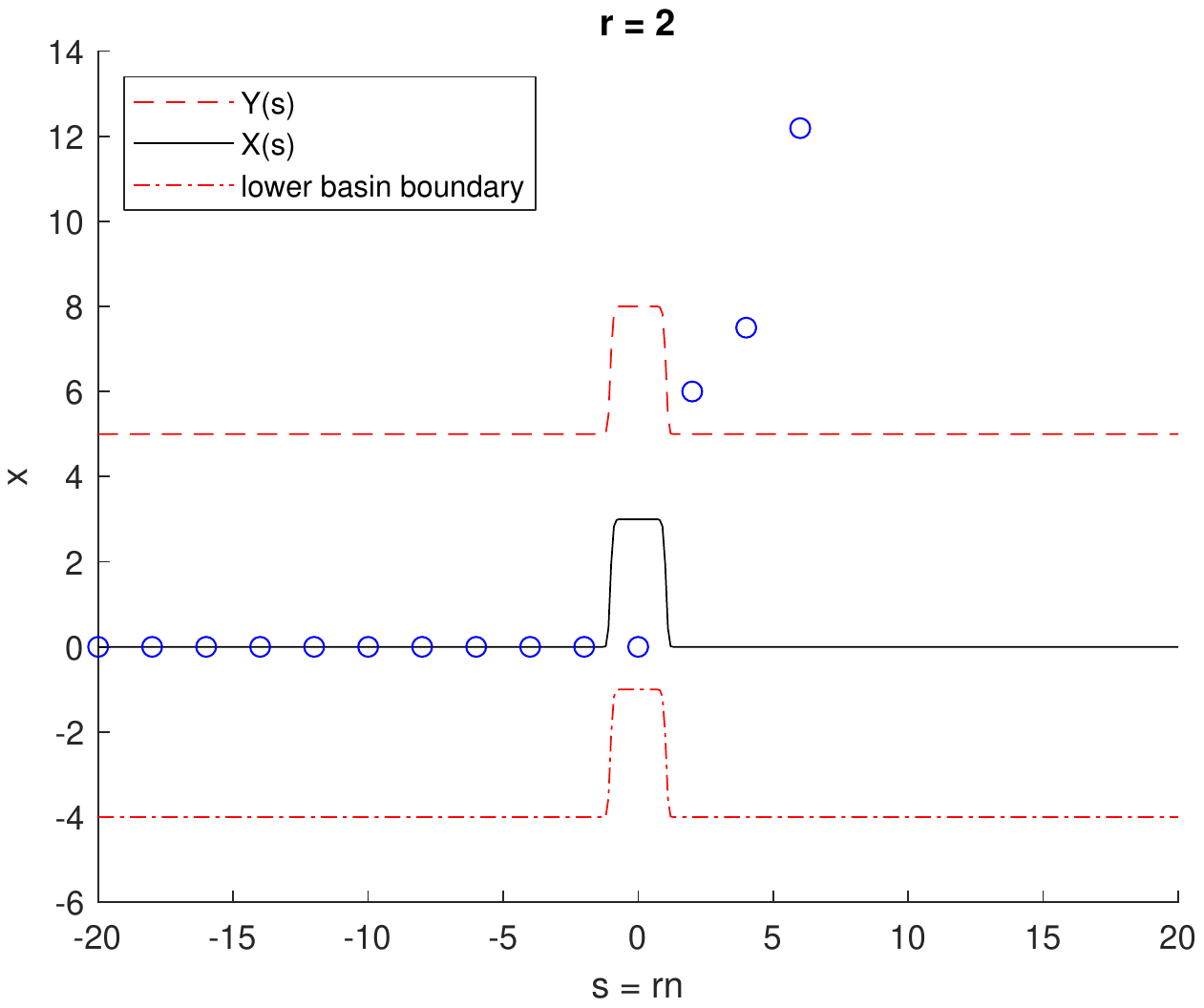}
\caption{The two red dashed curves together form the boundary of the basin of attraction for the stable path $(s,X(s))$. Even though $(s,X(s))$ is forward basin stable, the pullback attractor to $X_-$ does not endpoint track the path when $r = 2$.}
\label{fig:FBS}
\end{figure}

In flows, if $X_- \subset \B(Y_+,\lambda_+)$ for an attracting equilibrium $Y_+$, then the pullback attractor to $X_-$ will converge to $Y_+$ for all sufficiently large $r > 0$ (see Theorem 2 of \cite{kiers}). However, \cref{FBSnotEnough} shows that this is not the case in maps. We have $X_- \subset \B(X_+,\lambda_+)$ (in fact, $X_- = X_+ = 0$), but for all large values of $r$, the pullback attractor to $X_-$ does not endpoint track. If we let $\{x_n^\infty\}$ denote the pointwise limit of the $\{x_n^r\}$ as $r \to \infty$, then $x_{-1}^\infty = 0$, $x_0^\infty = 0$, $x_1^\infty = 6$, $x_2^\infty = 7.5$, etc. which diverges to infinity. The next Theorem gives a general result about what happens to the pullback attractor for large $r$ in maps.

\begin{thm}\label{largeR}
Let $Z_1 = f(X_-,\Lambda(0))$ and $Z_2 = f(Z_1,\lambda_+)$, and let $\{x_n^r\}$ be the pullback attractor to $X_-$. If $Z_2 \in \mathbb{B}(Y_+,\lambda_+)$ for some attracting fixed point $Y_+$, then $x_n^r \to Y_+$ as $n \to \infty$ for all sufficiently large $r>0$.
\end{thm}
\begin{proof}
By Lemma \ref{epsilonNeigh}, there exists an $r_1 > 0$ such that $|x_{-1}^{r} - X_-| < \epsilon_1$ for all $r \ge r_1$. If $\epsilon_1$ is small enough and $r$ is made larger, if necessary, then
$$|x_0^r - X_-| \le |f(x_{-1}^r, \Lambda(-r)) - f(X_-,\Lambda(-r))| + |f(X_-,\Lambda(-r))-f(X_-,\lambda_-)| < \epsilon_2$$
can be made as small as desired. By making $\epsilon_2$ small and $r$ larger if necessary,
$$|x_1^r - Z_1| = |f(x_0^r,\Lambda(0)) - f(X_-,\Lambda(0))| < \epsilon_3$$
can be made as small as desired. By making $\epsilon_3$ small and $r$ larger if necessary, we can ensure that
$$|x_2^r - Z_2| \le |f(x_1^r,\Lambda(r))- f(x_1^r,\lambda_+)| + |f(x_1^r,\lambda_+) - f(Z_1,\lambda_+)| < \epsilon_4$$
can be made small. For small $\epsilon_4$ and large $r$, \cref{compactK} implies that $x_n^r \to Y_+$ as $n \to \infty$.
\end{proof}

In the context of \cref{FBSnotEnough}, $X_- = 0$, $Z_1 = 6$, and $Z_2 = 7.5$. Loosely, we can think of $7.5$ as being in the basin of attraction of infinity when $\lambda= \lambda_+$, and indeed we find that $x_n^r$ diverges to infinity as $n \to \infty$ for all large $r$.

The above examples and discussion demonstrate some ways in which conditions for R-tipping in maps are different from those for flows (specifically, forward basin stability does not prevent R-tipping, and the behavior of the pullback attractor as $r \to \infty$ is different). However, some results between discrete and continuous dynamical systems are the same. We conclude this section with two results: one (\cref{canRTip}) giving conditions for when R-tipping is guaranteed to happen in maps, and one (\cref{inflowing}) in which R-tipping cannot happen. The corresponding results for flows are given in Theorem 2 and Proposition 1 of \cite{kiers}, respectively.

\begin{thm}\label{canRTip}
Suppose $(s,X(s))$ and $(s,Y(s))$ are two distinct stable paths, where $X(s)$ exists for all $s \in \R$ and $Y(s)$ for all sufficiently large $s$. If $X(u) \in \mathbb{B}(Y(v),\Lambda(v))$ for some $u < v$, then $(s,X(s))$ is not forward basin stable and there is a parameter shift $\tilde \Lambda$ such that there is R-tipping away from $X_-$ to $Y_+$ for some $r>0$.
\end{thm}
\begin{proof}
We will construct a parameter shift $\tilde \Lambda$ that gives R-tipping away from $X_-$. Since solutions to maps are only evaluated at integer multiples of $r$ and since $r$ is not yet determined, we introduce the functions
\begin{align*}
\sigma(s) & = s+u \\
\tau(s) & = s+v
\end{align*}
which will enable us to make sure the pullback attractor we introduce will be evaluated at both $u$ and $v$. Notice $\lim_{s \to \pm \infty} X(\sigma(s)), X(\tau(s)) = X_\pm$, $\lim_{s \to \pm \infty} Y(\sigma(s)), Y(\tau(s)) = Y_\pm$.

Now, let $x_n^{r,\sigma}$ denote the pullback attractor to $X_-$ under the parameter shift $\Lambda \circ \sigma$. Then
$$x_{n+1}^{r,\sigma} = f(x_n^{r,\sigma},\Lambda (\sigma(rn))).$$
Since $X(u) \in \B(Y(v),\Lambda(v))$, we can pick $\epsilon > 0$ such that $K = \overline{B_\epsilon(X(u))}$ satisfies $f(K,\Lambda(u)) = K_1 \subset \mathbb{B}(Y(v),\Lambda(v))$. Then by \cref{epsilonDistance} there is an $r_0 > 0$ such that for all $r \in (0,r_0)$, $|x_n^{r,\sigma} - X(\sigma(rn))| < \epsilon/2$ for all $n \in \Z$. By similar reasoning, we can find an $r_1 > 0$ such that for all $r \in (0,r_1)$, if $\{y_n\}$ is an orbit under the map $y_{n+1} = f(y_n,\Lambda(\tau(rn)))$ with $y_0 \in K_1$, then $y_n \to Y_+$ as $n \to \infty$. Now fix $r \in (0,\min\{r_0,r_1,v-u\})$. \par

We will construct a reparametrization
$$\tilde \Lambda(s) := \Lambda(\rho(s))$$
using a monotonic increasing function $\rho \in C^1(\R,\R)$ that increases rapidly from $\rho(s) = u$ to $\rho(s) = v$ but increases slowly otherwise. Define
$$\rho(s) = \left\{
\begin{array}{ll}
\sigma(s) & \text{ for } s \le 0 \\
\tau(s)-r & \text{ for } s \ge r
\end{array}
\right.$$
and then continue $\rho$ on $(0,r)$ in a $C^1$ way. \par

Let $\tilde x_n^r$ be the pullback attractor to $X_-$ with parameter shift $\tilde \Lambda$. By construction, we know that $\tilde x_0^r = x_0^{r, \sigma}$, so $|\tilde x_0^r - X(u)| < \epsilon/2$. Hence, $\tilde x_0^r \in K$. By our choice of $K$, we know
$$\tilde x_1^r = f(\tilde x_0^r,\tilde \Lambda(0)) = f(\tilde x_0^r, \Lambda(u)) \in K_1.$$
If we set $y_n = \tilde x_{n+1}^r$, then $y_0 \in K_1$ and $\{y_n\}_{n = 0}^\infty$ is a forward orbit under the map $y_{n+1} = f(y_n,\Lambda(\tau(rn)))$. Therefore, because $r < r_1$, $\tilde x_n^r \to Y_+$ as $n \to \infty$.
\end{proof}

The Ikeda map example in \cref{section:ikeda} demonstrates an application of \cref{canRTip}.

The idea of a {\em forward inflowing stable path} was introduced in \cite{kiers}, and it was proved that forward inflowing stability prevents R-tipping in flows of any dimension. Likewise, here we will show that there can be no R-tipping away from a forward inflowing stable path in a map. A stable path $(s,X(s))$ is {\em forward inflowing stable} if for each $s \in \R$ there exist compact sets $K(s)$ satisfying
\begin{enumerate}
\item if $s_1 < s_2$, then $K(s_1) \subset K(s_2)$;
\item $f(K(s),\Lambda(s)) \subset K(s)$ for all $s \in \R$;
\item $X_\pm \in \Int K_\pm$ where $K_- = \bigcap_{s \in \R} K(s)$ and $K_+ = \overline{\bigcup_{s\in \R} K(s)}$; and
\item $K_+ \subset \B(X_+,\lambda_+)$ is compact.
\end{enumerate}

\begin{thm}\label{inflowing}
Suppose $(s,X(s))$ is a forward inflowing stable path. Then there will not be R-tipping away from $X_-$.
\end{thm}
\begin{proof}
Fix $r > 0$ and let $\{x_n^r\}$ denote the pullback attractor to $X_-$. Let $\{K(s)\}$ be the sets guaranteed by forward inflowing stability. There exists an $\epsilon > 0$ such that $B_\epsilon(X_-) \subset K_-$, and by \cref{epsilonNeigh} there is some $N > 0$ such that $x_n^r \in B_\epsilon(X_-)$ for all $n \le -N$. Therefore, $x_n^r \in K_- \subset K(rn)$ for all $n \le -N$. It follows by induction that $x_n^r \in K(rn)$ for all $n \in \N$. In particular, $x_n^r \in K_+$, which is a compact subset of $\B(X_+,\lambda_+)$. Therefore, by \cref{compactK}, $x_n^r \to X_+$ as $n \to \infty$.
\end{proof}

\section{Flows and Maps}\label{flowsMaps}

In \cref{conditions} we gave some conditions for R-tipping in maps, some of which agree with related conditions for R-tipping in flows, some of which do not. In this section we will highlight the differences by exploring what happens when a map is obtained from a flow and there are two pullback attractors (one for the map and one for the flow).

Consider a continuous nonautonomous dynamical system of the form
\begin{align}\label{flow}
\frac{dx}{dt} = f(x,\Lambda(rt))
\end{align}
where $x \in \R^\ell$, $\Lambda: \R \to \R^m$ is $C^1$ and satisfies \cref{paramShift}, and $f: \R^\ell \times \R^m \to \R^\ell$ is $C^1$. The corresponding autonomous system is
\begin{align}\label{autoFlow}
\frac{dx}{dt} = f(x,\lambda)
\end{align}
for a fixed $\lambda \in \R^m$. Let $\Phi_\lambda$ denote the flow under \cref{autoFlow} for a given value of $\lambda$, so $\Phi_\lambda(t,x_0) = x(t)$ is the solution to \cref{autoFlow} with initial condition $x(0) = x_0$. Then we can define a map by evaluating the flow at integer time values, but allowing the parameter value to change with each time step:
\begin{align}\label{mapFromFlow}
x_{n+1} = \Phi_{\Lambda(rn)}(1,x_n).
\end{align}
Equation \cref{mapFromFlow} is of the same form as equation \cref{mapWithRate} because it is a nonautonomous map that depends on a time-varying parameter which changes according to \cref{paramShift}.

Suppose \cref{flow} has a stable path $(s,X(s))$ (all eigenvalues of $D_xf(X(s),\Lambda(s))$ have negative real part; see \cite{sebas}). Then $(s,X(s))$ is also a stable path for \cref{mapFromFlow}. For any $r > 0$, there is a unique pullback attractor $x^r(t)$ to $X_-$ under \cref{flow} (Theorem 2.2 of \cite{sebas}), and there is a unique pullback attractor $\{x_n^r\}$ to $X_-$ under \cref{mapFromFlow} (from \cref{uniqueSolution}).

In general it is {\em not} the case that $x_n^r = x^r(n)$ for all $n \in \Z$. To go from $x_n^r$ to $x_{n+1}^r$ in the map, we fix $\lambda = \Lambda(rn)$ and apply \cref{autoFlow} to $x_n^r$ for one time unit; over that whole time interval, $\lambda$ is fixed. On the other hand, to go from $x^r(n)$ to $x^r(n+1)$ in the flow, we apply \cref{flow} for one time unit, where $\Lambda(rn)$ can change over time. Since the parameter change affects the flow and the map differently, the pullback attractors can have different behavior. The following example illustrates this.

\begin{ex}\end{ex}
Let
$$\frac{dx}{dt} = -(x-\Lambda(rt))(x-\Lambda(rt) - 1)(x-\Lambda(rt)-2)$$
$$\Lambda(s) = \sech(s) + 0.4 \tanh(s)$$
and define the map
$$x_{n+1} = \Phi_{\Lambda(rn)}(1,x_n)$$
where $\Phi_\lambda$ is the flow generated by the continuous autonomous system. Then there are stable paths at $(s,X(s)) = (s,\Lambda(s)+2)$ and $(s,Z(s)) = (s,\Lambda(s))$ and an unstable path at $(s,Y(s)) = (s,\Lambda(s) + 1)$. \cref{fig:2Pullbacks} shows a plot of the two pullback attractors $x^r(t)$ and $\{x_n^r\}$ to $X_- = 1.6$ when $r = 3$. They have different behavior as time goes to infinity; $x^r(t)$ endpoint tracks the path $(s,X(s))$ while $\{x_n^r\}$ does not. (In fact, this is true for all sufficiently large $r$.)

\begin{figure}
\centering
\includegraphics[scale = 0.55]{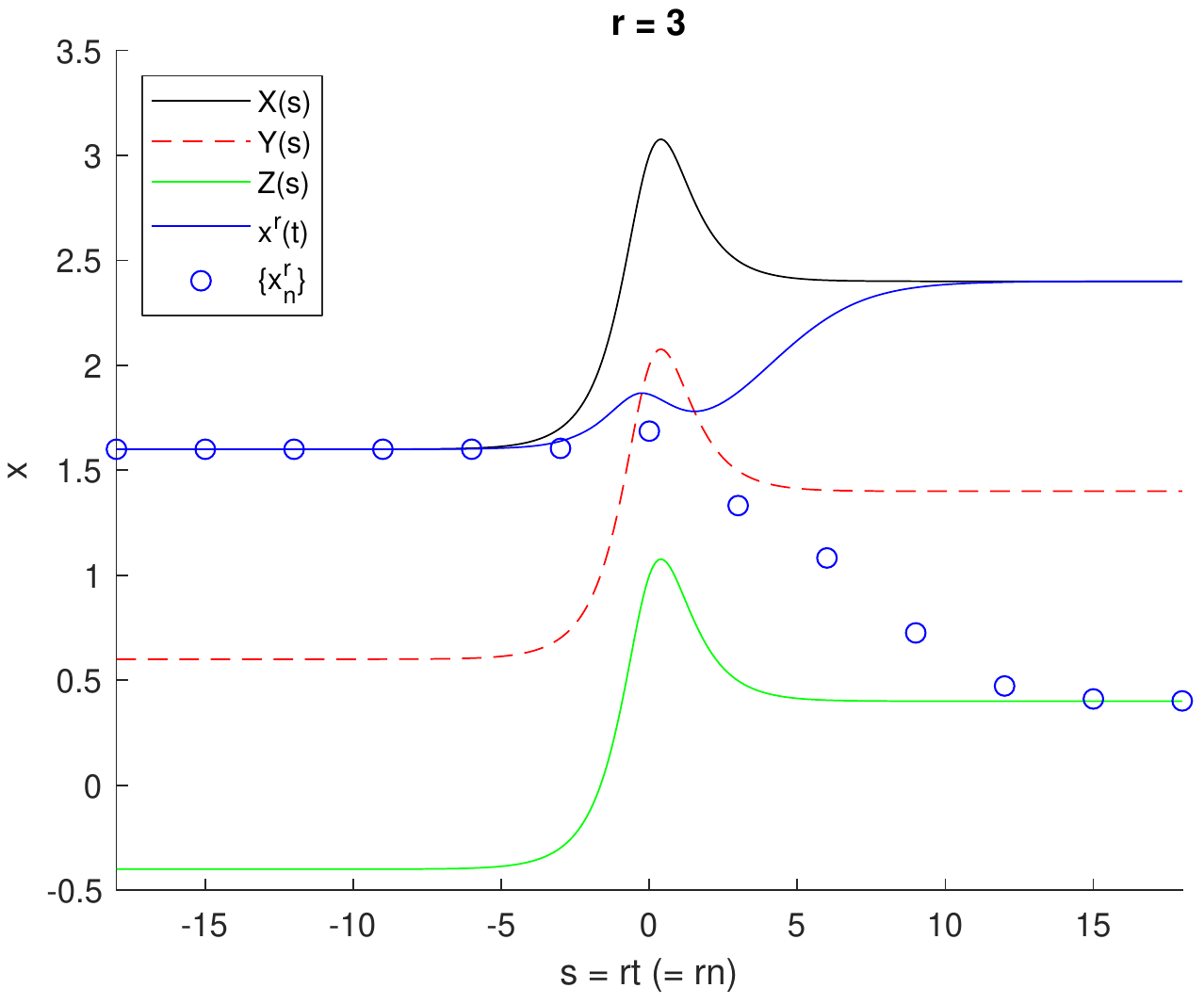}
\caption{The pullback attractor to $X_-$ in the continuous system, $x^r(t)$, endpoint tracks $(s,X(s))$, while the pullback attractor in the discrete system, $\{x_n^r\}$, tips to $Z_+$.}
\label{fig:2Pullbacks}
\end{figure}

\section{A 2-Dimensional Example: The Ikeda Map}\label{section:ikeda}

Up until now, all of our examples of rate-induced tipping in maps have involved 1-dimensional maps, even though our results have not specified anything about dimension. In part this is because it is easier to understand the complete dynamics in a 1-dimensional map; also it is easier to plot. In this section we will explore possibilities for rate-induced tipping in the Ikeda map, which is a 2-dimensional map that is used to model light in a ring cavity containing a dispersive nonlinear medium. The version of the map we will use here is the same as that used in \cite{ikedaMap}:
\begin{align}\label{complexMap}
z_n = a + R \exp\left[ i\left( \phi - \frac{p}{1 + |z_{n-1}|^2} \right) \right]z_{n-1},
\end{align}
where $z_j \in \C$ and $a, \phi, p , R \in \R$ are parameters. If we let $x_j = \Re(z_j)$ and $y_j = \Im(z_j)$, then we can rewrite \eqref{complexMap} as
\begin{equation}\label{realMap}
\begin{split}
x_n & = a + Rx_{n-1} \cos \theta - R y_{n-1} \sin \theta \\
y_n & = R x_{n-1} \sin \theta + R y_{n-1} \cos \theta.
\end{split}
\end{equation}
where $\theta = \phi - \frac{p}{1 + x_{n-1}^2 + y_{n-1}^2}$. As shown in \cite{ikedaMap}, there are generically an odd number of fixed points of \cref{realMap}. When there are three, let $X$, $Y$, and $Z$ denote the fixed point with the largest, middle, and smallest norm, respectively. Then $X$ is always stable, $Y$ is always a saddle point, and $Z$ is sometimes stable, sometimes a saddle. When $Z$ is stable, the 1-dimensional stable manifold for $Y$ splits the plane into basins of attraction for $X$ and $Z$. See \cref{fig:ikedaAuto} for an illustration.

\begin{figure}
\centering
\includegraphics[scale = 0.55]{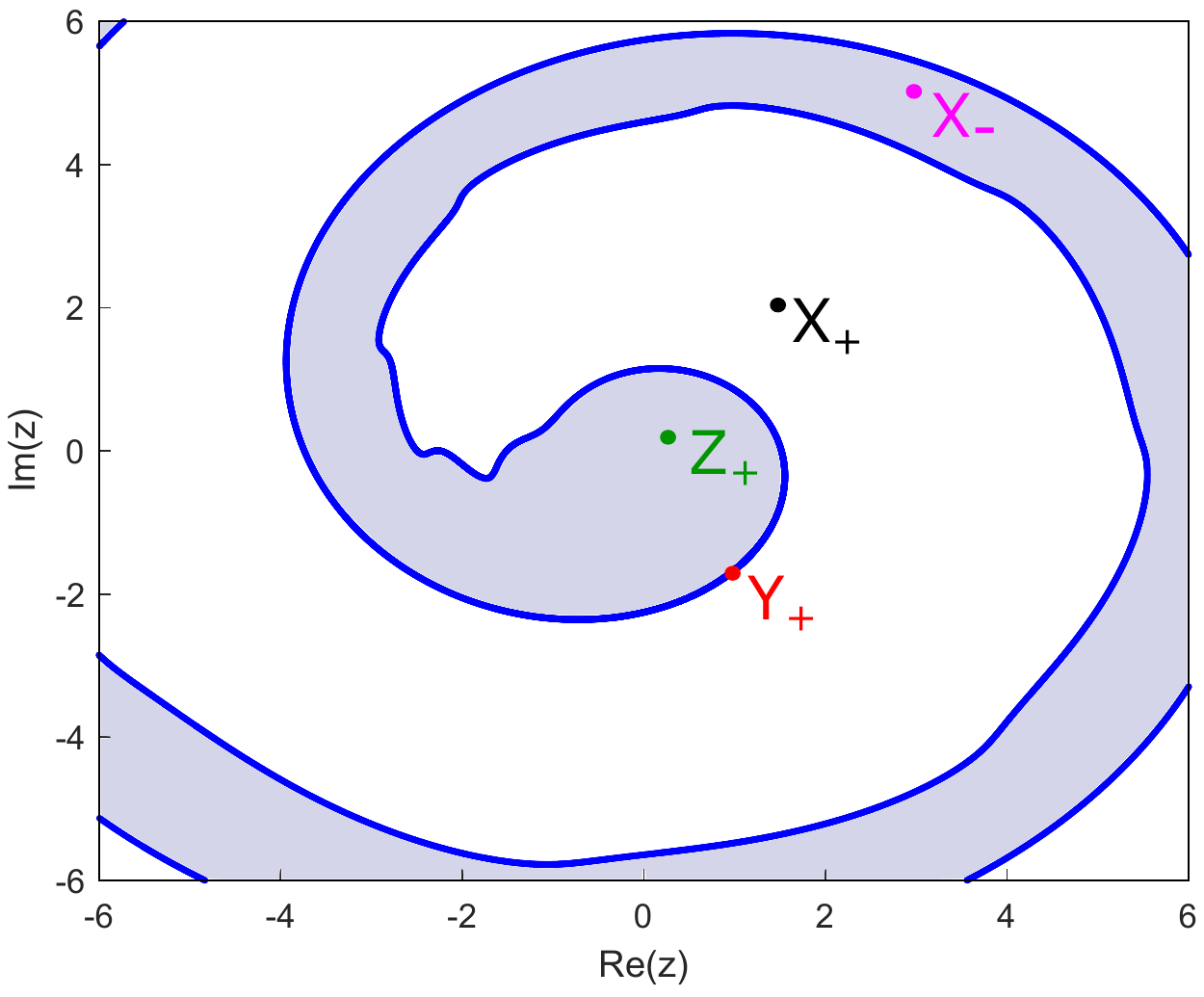}
\caption{Here $R = 0.9$, $\phi = 1$, $p = 6$, and $a = 0.5$. $Y_+$ is a saddle node, whose 1-dimensional stable manifold (plotted in blue) forms the boundary for the basins of attraction for stable fixed points $X_+$ (white) and $Z_+$ (gray). The point $X_- \approx (2.9698, 5.0274)$ is the largest-norm fixed point when $a = 4.5$; notice it is in the basin of attraction of $Z_+$.}
\label{fig:ikedaAuto}
\end{figure}

Of the four parameters in the map, $a$ is the one that is the most realistic to change, as it represents the amplitude of the inputted light. The other three parameters give information about the ring cavity itself and make sense to leave fixed. So, to get an example of R-tipping, we will fix $R = 0.9$, $\phi =1$, and $p = 6$, letting $a=a(s)$ be the time-varying parameter. From \cref{canRTip}, we know we can get R-tipping away from some $X_-$ if $X(u) \in \B(Z(v),a(v))$ for some $u < v$, where $(s,X(s))$ and $(s,Z(s))$ are distinct stable paths. In particular, this will be the case if $X_- \in \B(Z_+,a_+)$. If $a_+ = 0.5$ and $Z_+$ corresponds to the stable fixed point of \cref{realMap} with smallest norm, then we already know what $\B(Z_+,a_+)$ looks like from \cref{fig:ikedaAuto}. So, we just need to find a value for $a_-$ such that the fixed point with largest norm, $X_-$, is in $\B(Z_+,a_+)$. One such value is $a_- = 4.5$ ($X_- \approx (2.9698, 5.0274)$; see \cref{fig:ikedaAuto}). Thus, we will decrease $a(s)$ from $a_- = 4.5$ to $a_+ = 0.5$ according to the function
$$a(s) = \frac{5}{2} - 2 \tanh(s).$$
The bifurcation diagram for this parameter change is shown in \cref{fig:tipIkeda}. There is one stable path of fixed points for all time, $(s,X(s))$, which corresponds to the fixed point with largest norm. When $a < 1.6$, there are two other paths, $(s,Y(s))$ and $(s,Z(s))$, which are the unstable middle-norm and stable smallest-norm fixed points, respectively. Since we have carefully set up our parameter shift in a way that satisfies the conditions of \cref{canRTip}, we hope to see R-tipping away from $X_-$ to $Z_+$ for some rate $r>0$, which indeed we do when $r = 2$; see \cref{fig:tipIkeda}.

\begin{figure}
	\centering
	\begin{subfigure}{0.4 \textwidth}
		\centering
		\includegraphics[width = \textwidth]{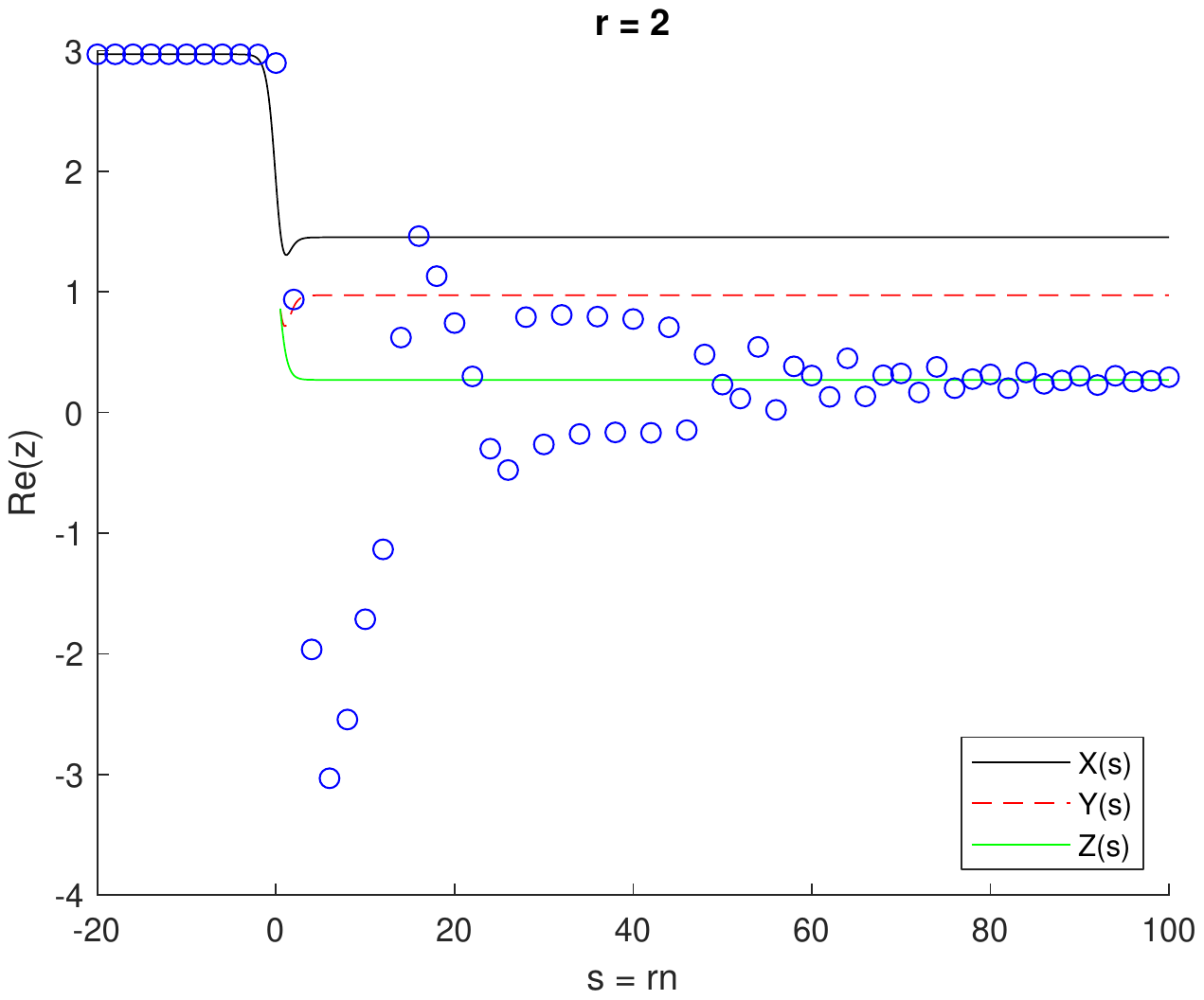}
	\end{subfigure}
	\hspace{0.5in}
	\begin{subfigure}{0.4 \textwidth}
		\centering
		\includegraphics[width = \textwidth]{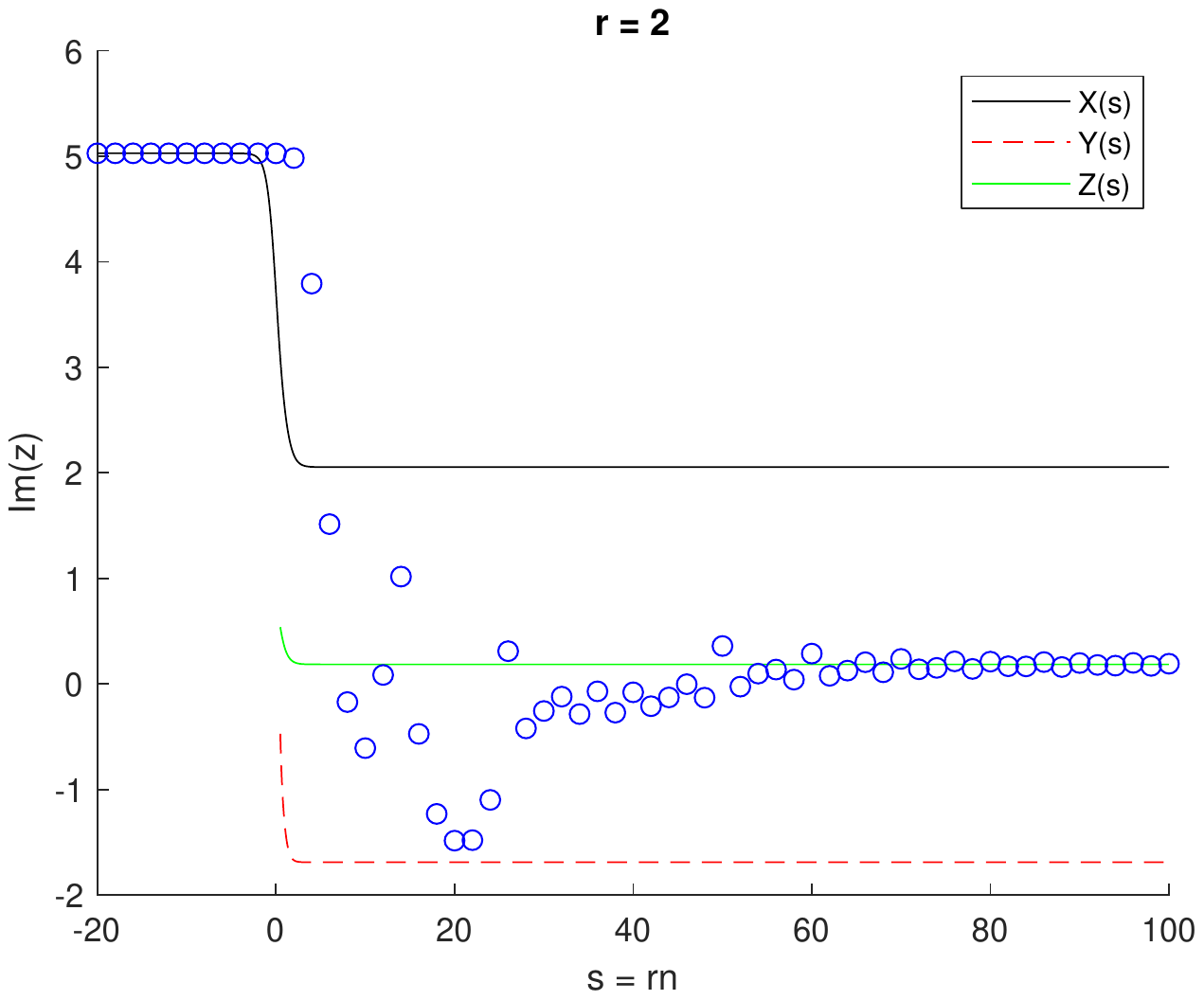}
	\end{subfigure}
	\caption{When $r = 2$, the pullback attractor to $X_-$ tips to $Z_+$.}
	\label{fig:tipIkeda}
\end{figure}

\section*{Acknowledgments}
The author is greatly indebted to Chris Jones for his guidance on this project, in pointing out the Ikeda map and for his suggestions on some proof techniques.


\begin{thebibliography}{}
%
%

\bibitem[1]{sebas} P. Ashwin, C. Perryman, and S. Wieczorek. Parameter shifts for nonautonomous systems in low dimension: bifurcation- and rate-induced tipping. {\em Nonlinearity}, 30(6), 2017.

\bibitem[2]{pete} P. Ashwin, S. Wieczorek, R. Vitolo, and P. Cox. Tipping points in open systems: bifurcation, noise-induced and rate-dependent examples in the climate system. {\em Philosophical Transactions of the Royal Society A}, 370, 2012.

\bibitem[3]{fenichel} N. Fenichel. Persistence and smoothness of invariant manifolds for flows. {\em Indiana University Mathematics Journal}, 21(3), 1971.

\bibitem[4]{hahn} J. Hahn. {\em Rate-Dependent Bifurcations and Isolating Blocks in Nonautonomous Systems}. PhD thesis, University of Minnesota, 2017.

\bibitem[5]{ikedaMap} S. M. Hammel, C. K. R. T. Jones, and J. V. Moloney. Global dynamical behavior of the optical field in a ring cavity. {\em Journal of the Optical Society of America B}, 2(4), 1985.

\bibitem[6]{horn} R. A. Horn and C. R. Johnson. {\em Matrix Analysis}. Cambridge University Press, 2nd edition, 2013.

\bibitem[7]{kiers} C. Kiers and C. K. R. T. Jones. On conditions for rate-induced tipping in multi-dimensional dynamical systems. {\em Journal of Dynamics and Differential Equations}, 2019. 

\bibitem[8]{nonauto} P. E. Kloeden, C. P\"otzsche, and M. Rasmussen. Discrete-time nonautonomous dynamical systems. In R. Johnson and M. P. Pera, editors, {\em Stability and Bifurcation Theory for Non-Autonomous Differential Equations}, chapter 2, pages 35-102. Springer-Verlag Berlin Heidelberg, 2013.

\bibitem[9]{kuehn} C. Kuehn. {\em Multiple Time Scale Dynamics}, volume 191 of {\em Applied Mathematical Sciences}. Springer International Publishing, 2015.

\bibitem[10]{noiseAndRate} P. Ritchie and J. Sieber. Probability of noise- and rate-induced tipping. {\em Physical Review E}, 95(5), 2017.

\bibitem[11]{compost} S. Wieczorek, P. Ashwin, C. Luke, P. Cox. Excitability in ramped systems: the compost-bomb instability. {\em Proceedings of the Royal Society of London A}, 467, 2011.

\end{thebibliography}

\end{document}